%% file: Boundarypaper.tex
\documentclass[12pt,a4paper,twoside]{article}
\usepackage{graphicx}

\usepackage{authblk}

\input{new_commands.tex}

\title{Perfect Conductor Boundary Conditions for Geometric Particle-in-Cell Simulations of the Vlasov--Maxwell System in Curvilinear Coordinates\thanks {\textbf{Funding}: This work has been carried out within the framework of the EUROfusion Consortium and has received funding from the Euratom research and training programme 2014-2018 and 2019-2020 under grant agreement No 633053. The views and opinions expressed herein do not necessarily reflect those of the European Commission. B.~P.~has also been
supported by Deutsche Forschungsgemeinschaft (DFG) through the TUM International Graduate School
of Science and Engineering (IGSSE).}}

\author[1,2]{Benedikt Perse}
\author[3,1]{Katharina Kormann}
\author[1,2]{Eric Sonnendr{\"u}cker}
\affil[1]{Max Planck Institute for Plasma Physics, Boltzmannstr.~2, 85748 Garching, Germany (\texttt{\{benedikt.perse, katharina.kormann, eric.sonnendruecker\}@ ipp.mpg.de}).}
\affil[2]{Department of Mathematics, Technical University of Munich, Boltzmannstr.~3, 85748 Garching, Germany.}
\affil[3]{Department of Information Technology, Uppsala University, Box 337, 751 05 Uppsala, Sweden (\texttt{katharina.kormann@it.uu.se}).}

\begin{document}

\maketitle

\begin{abstract}
Structure-preserving methods can be derived for the Vlasov--Maxwell system from a discretisation of the Poisson bracket with compatible finite-elements for the fields and a particle representation of the distribution function. These geometric electromagnetic particle-in-cell (GEMPIC) discretisations feature excellent conservation properties and long-time numerical stability. This paper extends the GEMPIC formulation in curvilinear coordinates to realistic boundary conditions. We build a de Rham sequence based on spline functions with clamped boundaries and apply perfect conductor boundary conditions for the fields and reflecting boundary conditions for the particles. The spatial semi-discretisation forms a discrete Poisson system. Time discretisation is either done by Hamiltonian splitting---yielding a semi-explicit Gauss conserving scheme---or by a discrete gradient scheme applied to a Poisson splitting---yielding a semi-implicit energy-conserving scheme. Our system requires the inversion of the spline finite element mass matrices, which we precondition with the combination of a Jacobi preconditioner and the spectrum of the mass matrices on a periodic tensor product grid.
\end{abstract}

\input{papertext}

\input{Boundarypaper.bbl}
\end{document}

%% file: new_commands.tex
\usepackage{amssymb,amsmath,amsfonts,amstext,amsthm}
\usepackage{bbm}
\usepackage{latexsym}
\usepackage{multirow}
\usepackage[english]{babel}
\usepackage{epstopdf}
\usepackage{subcaption}

\usepackage[bbgreekl]{mathbbol}

\usepackage{tikz}
\usepackage{tikz-cd}
\usetikzlibrary{arrows,matrix,shapes,positioning}
\usetikzlibrary{calc,patterns,shapes.geometric}
\usetikzlibrary{decorations.markings,decorations.pathmorphing,decorations.pathreplacing}
\usetikzlibrary{decorations.text,decorations.shapes}

\usepackage{graphicx}
\usepackage{xcolor}
\usepackage{soul}
\usepackage[colorlinks=true,breaklinks=true]{hyperref}
\usepackage{url}

\let\div\undefined
\DeclareMathOperator{\div}{div}
\DeclareMathOperator{\curl}{curl}

\newcommand{\du}{\,\mathrm{d}}
\newcommand{\R}{\mathbb{R}}

\newcommand{\xb}{{\mathbf{x}}}
\newcommand{\vb}{{\mathbf{v}}}

\newcommand{\Vb}{{\mathbf{V}}}

\newcommand{\Bb}{{\mathbf{B}}}

\newcommand{\Eb}{{\mathbf{E}}}
\newcommand{\Fb}{{\mathbf{F}}}
\newcommand{\Gb}{{\mathbf{G}}}

\newcommand{\Sb}{{\mathcal{S}}}

\newcommand{\bt}{\mathbf{t}}
\newcommand{\bn}{\mathbf{n}}

\newcommand{\phib}{{\boldsymbol{\phi}}}
\newcommand{\xib}{{\boldsymbol{\xi}}}
\newcommand{\Xib}{{\boldsymbol{\Xi}}}

\newcommand{\tub}{\tilde{\mathbf{u}}}
\newcommand{\teb}{{\tilde{\mathbf{e}}}}

\newcommand{\tE}{{\tilde{E}}}

\newcommand{\tjb}{{\tilde{\mathbf{j}}}}
\newcommand{\tEb}{{\tilde{\mathbf{E}}}}
\newcommand{\tBb}{{\tilde{\mathbf{B}}}}
\newcommand{\tAb}{{\tilde{\mathbf{A}}}}
\newcommand{\tPhi}{{\tilde{\Phi}}}

\newcommand{\tab}{\tilde{\mathbf{a}}}
\newcommand{\tbb}{{\tilde{\mathbf{b}}}}
\newcommand{\tb}{{\tilde{b}}}
\newcommand{\tB}{{\tilde{B}}}

\newcommand{\trho}{{\tilde{\rho}}}
\newcommand{\trhob}{{\tilde{\boldsymbol{\rho}}}}
\newcommand{\tphib}{{\tilde{\boldsymbol{\phi}}}}
\newcommand{\tf}{\tilde{f}}
\newcommand{\tJb}{\tilde{\mathbf{J}}}
\newcommand{\Jb}{\mathbf{J}}
\newcommand{\tvpb}{\boldsymbol{\tilde{\varphi}}}
\newcommand{\tvp}{\tilde{\varphi}}

\newcommand{\WM}{\ensuremath{\mathbb{W}}}

\newcommand{\JJ}{\ensuremath{\mathbb{J}}}
\newcommand{\tM}{\ensuremath{\mathsf{\tilde{M}}}}
\newcommand{\C}{\ensuremath{\mathsf{C}}}
\newcommand{\D}{\ensuremath{\mathsf{D}}}
\newcommand{\G}{\ensuremath{\mathsf{G}}}
\newcommand{\tBB}{\ensuremath\tilde{{\mathbb{B}}}}

\newcommand{\NN}{\ensuremath{\mathbb{N}}}

\newcommand{\tom}{\ensuremath{\tilde{\Omega}}}
\newcommand{\tLa}{\ensuremath{\tilde{\Lambda}}}

\newcommand{\tLab}{\ensuremath{\tilde{\boldsymbol{\Lambda}}}}
\newcommand{\tLaB}{\ensuremath{\tilde{\mathbf{\Lambda}}}}
\newcommand{\tLaBB}{\ensuremath{\tilde{\mathbb{\Lambda}}}}

\newcommand{\kb}{\mathbf{k}}

\newtheorem{theorem}{Theorem}[section]
\newtheorem{prop}[theorem]{Proposition}
\newtheorem{mydef}[theorem]{Definition}

\numberwithin{equation}{section}

%% file: papertext.tex
\section{Introduction}
In \cite{perse2021geometric}, we have introduced a general coordinate transformation into the geometric electromagnetic particle-in-cell (GEMPIC) framework \cite{kraus2016gempic}. However, we assumed periodic boundary conditions, which also limits the coordinate transformation to periodic mappings. In this paper, we extend the GEMPIC framework to more realistic boundary conditions enabling the use of radial grids such as cylindrical or elliptical grids. In our numerical experiments, we focus on perfect conductor boundary conditions as described in \cite{guo1993global} to model a lossless metallic surface and for the particles, we assume specular reflection.

\subsection{Related work}
In 2012, Fichtl, Finn \& Cartwright \cite{fichtl2012arbitrary} proposed an electrostatic 2D2V code, where the fields are discretised with the finite difference method and the particle position and velocity are both pushed in logical coordinates. The code is momentum conserving and applies homogeneous Neumann boundary conditions for the fields.
Likewise, Delzanno et al. \cite{delzanno2013cpic} described an electrostatic 3D PIC code called CPIC with finite difference discretisation. The code uses a hybrid particle pusher and allows for mesh refinement. Absorbing and reflecting particle boundaries are tested on a sinusoidally distorted grid and on an annulus. For the fields Neumann or Dirichlet boundary conditions can be applied.

In \cite{chacon2019}, Chen \& Chac\'on have added perfect conductor boundary conditions and a reflecting particle boundary to their electromagnetic 2D3V PIC code in curvilinear coordinates for the Vlasov--Darwin model.

More recently, Xiao \& Qin \cite{xiao2020explicit} extended their geometric PIC code \cite{xiao2018structure} to orthogonal curvilinear coordinate transformations maintaining the explicit time discretisation via a Hamiltonian splitting. The explicit time splitting is obtained by using a logical velocity variable. However, it is only possible for the special case of an orthogonal transformation. They apply perfect electric conductor boundary conditions in two directions and periodicity in the third direction and particles hitting the boundary are removed from the simulation.

A different approach was taken by Wang, Qin, Sturdevant \& Chang \cite{wang2020geometric} using the structure-preserving framework to build an electrostatic 2D PIC code on unstructured grids with fully kinetic ions and adiabatic electrons. A de Rham complex is constructed with Whitney forms assuming homogeneous Dirichlet boundary conditions for the fields and a reflecting boundary for the particles. This setup allows for simulations of ion Bernstein waves in a 2D magnetized plasma.

Apart from PIC methods, Colella, Dorr, Hittinger \& Martin \cite{colella2011high} used a finite volume discretisation of the Vlasov--Poisson system with Dirichlet boundary conditions to perform simulations in a D-shaped annular geometry. In \cite{mccorquodale2015high}, this code was extended to mapped multi block grids and tested on a sinusoidally distorted mesh with advection problems. Vogman et al. \cite{vogman2018conservative} introduced a continuum code with finite volume discretisation for an electrostatic axisymmetric cylindrical Vlasov--Poisson system using specular reflection as particle boundary conditions and Dirichlet boundary conditions for the fields. 

\subsection{Outline}
This paper is structured in the following way:
In Section \ref{sec:VMcurvilinear}, the coordinate transformation is introduced into the Vlasov--Maxwell system. First, we examine the natural boundary conditions of the weak formulation of the Vlasov--Maxwell system. Then, the notation for curvilinear coordinate transformations is introduced and applied to the electromagnetic fields.
Section \ref{sec:spatial_discretiastion} approaches the structure-preserving semi-discretisation of the Vlasov--Maxwell system. First, we construct  conforming spline finite elements with boundary conditions that form a discrete de Rham sequence in logical coordinates. Then, we represent the electromagnetic fields in the spline basis and apply perfect conductor boundary conditions. For the particles, reflecting boundary conditions are considered. Last, we derive the equations of motion coming from  the semi-discrete Poisson structure for perfect conductor boundary conditions. 
In Section \ref{sec:temporal_discretisation}, we discretise the equations of motion in time and briefly comment on the handling of the particle boundary conditions.
Section \ref{sec:preconditioner} discusses preconditioners for the conjugate gradient solvers of the mass matrices. 
The implementation of the boundary conditions is verified in a numerical test case with various coordinate transformations in Section \ref{sec:numericsboundary}.
Section \ref{sec:conclusion} concludes the paper with a short summary and an outlook to future work.

\section{The Vlasov--Maxwell System in Curvilinear Coordinates} \label{sec:VMcurvilinear}

\subsection{The Vlasov--Maxwell System}
The Vlasov equation in physical phase-space coordinates $(\xb,\vb)$ for a species $s$ with charge $q_s$ and mass $m_s$ reads
{\small
\begin{align}\label{V}
&\frac{\partial f_s(\xb,\vb,t)}{\partial t}+ \vb \cdot \nabla_{\xb} f_s(\xb,\vb,t)+\frac{q_s}{m_s} (\Eb(\xb,t)+\vb\times \Bb(\xb,t))\cdot \nabla_\vb f_s(\xb,\vb,t)=0,
\end{align} }
where $\Eb$ and $\Bb$ denote the electromagnetic fields, which are evolved according to Maxwell's equations,
\begin{subequations}
\begin{align} \label{m1}
\frac{\partial \Eb(\xb,t)}{\partial t} &= \nabla_{\xb} \times \Bb(\xb,t) - \Jb(\xb,t),
\\ \label{m2}
\frac{\partial  \Bb(\xb,t)}{\partial t}&=-\nabla_{\xb} \times \Eb(\xb,t),
\\ \notag
\nabla_{\xb} \cdot \Eb(\xb,t)&=\rho(\xb,t),
\\ \notag
\nabla_{\xb} \cdot  \Bb(\xb,t) &=0.
\end{align} 
\end{subequations}
The system couples through the first two moments of the particle distribution function $f_s$,  the charge and the current densities,
\begin{align*}
\rho(\xb,t)=\sum_s q_s \int f_s(\xb,\vb,t) \du\vb, \ \Jb(\xb,t)=\sum_s q_s \int f_s(\xb,\vb,t) \vb \du\vb.
\end{align*}

The system is closed by suitable initial and boundary conditions.

For the solution of Maxwell's equation, we consider a mixed form with a weak description of Amp{\`e}re's and electric Gauss' law. Then, we have chosen the electric field to be an differential 1-form, $\Eb \in H(\curl,\Omega)$ and the magnetic field to be a differential 2-form,$\Bb \in H(\div,\Omega)$. Let $\boldsymbol{\varphi} \in H(\curl,\Omega)$ and $\psi \in H^1(\Omega)$ be test functions. Then, we obtain the following equations using the cross product form of the Divergence theorem \eqref{crossdivergencethm}:
\begin{subequations}
\begin{align}\label{weakformula1} 
\int_{\Omega} \boldsymbol{\varphi} \cdot \frac{\partial \Eb}{\partial t} \du\xb &= \int_{\Omega} \nabla_\xb \times \boldsymbol{\varphi} \cdot \Bb   \du\xb+\int_{\partial \Omega} (\Bb \times \boldsymbol{\varphi}) \cdot \bn \du\sigma-\Jb^\star(\boldsymbol{\varphi})(t),
\\ \label{Faraday}
\frac{\partial  \Bb}{\partial t}&=-\nabla_{\xb} \times \Eb,
\\ \label{weakformula2} \int_{\Omega} \nabla \psi \cdot \Eb  \du\xb &= \int_{\partial\Omega} \psi (\Eb \cdot \bn)  \du\sigma-\rho^\star( \psi)(t),
\\ \label{mGauss}
\nabla_{\xb} \cdot  \Bb &=0,
\end{align} 
\end{subequations}
where $\Jb^\star\in {H}^\star(\div,\Omega)$ and $\rho^\star \in {L^2}^\star(\Omega)$ are  linear functionals defined as 
$\Jb^\star(\boldsymbol{\varphi})(t)= \langle \boldsymbol{\varphi}, \Jb \rangle_{L^2}, \rho^\star( \psi)(t) = \langle \psi, \rho \rangle_{L^2} $ and $\bn$ denotes the unit outer normal.

In order to close the system, we need boundary conditions both for the Vlasov and the field equations. For realistic geometry in plasma physics, we assume a D-shaped toroidal mesh. This can be modeled by a coordinate system with one radial and two angular directions. In the next section, we will introduce a coordinate transformation, which would be typically used for a toroidal coordinate system. Then, we assume periodic boundary conditions for the two angular dimensions and some form of Dirichlet or Neumann boundary conditions for the radial direction. If the boundary is defined by the wall of the device, perfect conductor boundary conditions are most suitable.

\subsection{Coordinate Transformation}
We introduce a coordinate transformation with the goal of deriving a curvilinear description of the Vlasov--Maxwell system. 
First, we present our notation for the curvilinear coordinates before discussing how these can be consistently combined with differential forms. We consider a bijective coordinate transformation from the logical space $\tom:=[0,1]^3$ to the physical space $\Omega$, e.g.~a Torus in spherical coordinates. The transformation map is denoted by 
\begin{align*}
F \colon \tom \to \Omega \subset \R^3, \ \xib \mapsto
F(\xib)=\xb,
\end{align*}
 where $\xib=(
\xi_1,\xi_2,\xi_3)^\top, \xb=(x_1,x_2,x_3)^\top $ are the variables on the logical and physical mesh, respectively.

The matrix of the partial derivatives, the Jacobian matrix, and its determinant are defined as
\begin{align*}
\left(DF(\xib)\right)_{ij}&=\frac{\partial F_i(\xib)}{\partial \xi_j}=\frac{\partial x_i}{\partial \xi_j},\\
J_F(\xib)&=\det(DF(\xib)).
\end{align*}
We assume that the mapping is non-singular, i.e.~$J_F(\xib)\neq 0  \ \forall  \xib\in \tom$, and therefore, the Jacobian matrix is invertible. 
\begin{mydef}\label{def:derivativevectors}
The column vectors of the Jacobian matrix form the so-called covariant basis of the tangent space, 
\begin{align*}
\bt_i=\frac{\partial F(\xib)}{\partial \xi_i}=\frac{\partial \xb}{\partial \xi_i}, \ DF=(\bt_1|\bt_2|\bt_3)
\end{align*}
whereas the columns of the transposed inverse Jacobian matrix form the dual basis, which is called the contravariant basis of the cotangent space,
\begin{align*}
\bn_i=\frac{\partial \xi_i}{\partial \xb}, \ DF(\xib)^{-\top}=: N(\xib)=(\bn_1|\bn_2|\bn_3).
\end{align*} 
\end{mydef}

\begin{prop}\label{prop:bv}
The following relations hold between the covariant and the contravariant basis vectors: 
\begin{align*}
 \bn_1&= \frac{1}{J_F}  \bt_2 \times  \bt_3,& \bn_2 \times \bn_3 &= \frac{1}{J_F}  \bt_1,\\
 \bn_2&= \frac{1}{J_F}  \bt_3 \times  \bt_1,& \bn_3 \times \bn_1 &= \frac{1}{J_F} \bt_2, \\
 \bn_3&= \frac{1}{J_F}  \bt_1 \times  \bt_2, & \bn_1 \times \bn_2 &= \frac{1}{J_F}  \bt_3.
\end{align*}
\end{prop}
Note that $\bn$ without an index refers to the unit outer normal vector while $\bn_{1/2/3}$ refers to the columns of $N$.
\subsection{Transformation of Differential Forms}
We introduce curvilinear coordinates to the differential forms and show how they are transformed in a consistent way as can be seen in \cite{kreeft2011mimetic}.

\begin{mydef}
For a scalar differential 0-form, $ g\in H^1(\Omega)$, we define $\tilde{g}\in H^1(\tom)$ as
\begin{align}\label{0form}
\tilde g(\xib):= g(F(\xib))=g(\xb).
\end{align}
\end{mydef}
Next, we consider the transformation of the other differential forms.
\begin{prop} \label{prop:derivtrafo} We have the following properties:
\begin{enumerate}
\item A vector function, $\Eb\in H(\curl,\Omega)$, corresponding to a differential 1-form, is transformed by the covariant Piola transform,
\begin{align}\label{PiolaE}
\Eb(\xb)=N(\xib) \tEb(\xib)\text{ with } \tEb \in H(\curl,\tom).
\end{align}
\item A vector function,  $\Bb\in H(\div,\Omega)$, corresponding to a differential 2-form, is transformed by the contravariant Piola transform,
\begin{align}\label{PiolaB}
\Bb(\xb)=\frac{DF(\xib)}{J_F(\xib)}\tBb(\xib) \text{ with } \tBb \in H(\div,\tom).
\end{align}
\item A scalar differential 3-form, $h \in L^2(\Omega)$, is related to $\tilde h \in L^2(\tom)$ via
\begin{align*}
h(\xb)=\frac{1}{J_F(\xib)}\tilde{h}(\xib).
\end{align*}  
\end{enumerate}
\end{prop}

Next, let us formulate the Vlasov--Maxwell system under a coordinate transformation with two periodic (without loss of generality $\xi_2$ and $\xi_3$) and one radial direction ($\xi_1$). We note that the normal to the boundary vectors in logical coordinates are then given by the (negative) unit vector along the first direction.

\begin{prop}\label{prop:transformedVM}
Under the coordinate transformation $F(\xib)=\xb$ with periodic boundaries along $\xi_2$, $\xi_3$
\begin{enumerate}
\item the Vlasov equation \eqref{V} transforms to
{\small
\begin{align*}
\frac{\partial \tf_s(\xib,\vb,t)}{\partial t}+& N(\xib)^\top \vb  \cdot \nabla_{\xib} \tf_s(\xib,\vb,t)
\\+&\frac{q_s}{m_s} N(\xib) \left(\tEb(\xib,t)+( N(\xib)^\top\vb)\times \tBb(\xib,t)\right)\cdot \nabla_\vb \tf_s(\xib,\vb,t)=0;
\end{align*} }
\item Faraday's \eqref{Faraday} and magnetic Gauss' laws \eqref{mGauss} in strong form do not change, i.e.~
\begin{subequations}
\begin{align}\label{strongFaraday}
 \frac{\partial \tBb(\xib,t)}{\partial t}&=-  \nabla_{\xib} \times \tEb(\xib,t), 
 \\ \label{strongmGauss}
 \nabla_\xib \cdot \tBb(\xib,t) &=0; 
\end{align}
\end{subequations}
\item the weak formulation of Amp\`ere's \eqref{weakformula1} and Gauss' laws \eqref{weakformula2} is transformed for all $\tilde{\boldsymbol{\varphi}}\in H(\operatorname{curl},\tom),$ $\tilde \psi \in H^1(\tom)$ as
\begin{subequations}
{\small
\begin{align}\label{weakAmpere}
\frac{\partial}{\partial t}\int_{\tom}  N \tilde{\boldsymbol{\varphi}} \cdot N \tEb |J_F| \du\xib =& \int_{\tom}   \frac{DF}{J_F} \nabla_\xib \times \tilde{\boldsymbol{\varphi}} \cdot \frac{DF}{J_F} \tBb |J_F| \du\xib \\ \notag
+&\int_{\partial \tilde \Omega} \left(-\tvp_2 \frac{\bt_3}{J_F} +\tvp_3 \frac{\bt_2}{J_F}\right) \cdot DF \tBb \du\tilde{\sigma} \\ \notag -&\int_{\tom} N \tilde{\boldsymbol{\varphi}} \cdot N \tJb |J_F| \du \xib,
\\\label{weakGauss}
\int_{\tom} N \nabla_\xib \tilde \psi \cdot N \tEb |J_F| \du\xib =&  \int_{\partial\tilde \Omega} \tilde \psi  (N \tEb \cdot \bn_1)   J_F \du\tilde{\sigma}-\int_{\tom} \tilde{ \psi} \trho |J_F| \du\xib.  
\end{align} }
\end{subequations}
\end{enumerate}
\end{prop}

\begin{proof}
The proof for vanishing boundary conditions can be found in \cite[Proposition 2.5]{perse2021geometric}. Therefore, we only focus on the boundary parts in the weak formulation of Amp\`ere's and Gauss' laws. 

Since the normal vector in logical coordinates is given as $\tilde \bn= \pm (1,0,0 )^\top = \pm \hat{\textbf{e}}_x$ and the normal vector transforms with the covariant Piola transform \eqref{PiolaE}, it takes the following form  in physical coordinates: 
\begin{align*}
\bn = \frac{N(\xib) \tilde{\bn} }{\| N(\xib) \tilde{\bn}\|} = \pm \frac{\bn_1}{\|\bn_1\|}.
\end{align*}  

To compute the boundary part of Amp\`ere's law tested with $\boldsymbol{\varphi}$ in curvilinear coordinates, we insert the Piola transforms for the electromagnetic fields and the test function, \eqref{PiolaE} and \eqref{PiolaB}, and use $\du\sigma=\|\bt_2 \times \bt_3\| \du\tilde{\sigma}$ with $\du \tilde{\sigma}=\big|_{\xi_1=0}^{1} \du \xi_2 \du \xi_3$,   
\begin{align*}
\int_{\partial \tilde \Omega}  \left(\frac{DF}{J_F} \tBb \times N \tvpb \right) \cdot N \tilde{\bn} \frac{\| \bt_2 \times \bt_3\|}{\| N \tilde{\bn}\|} \du\tilde{\sigma}.
\end{align*}
Since the scalar triple product permutes,  it holds that $$\left(\frac{DF}{J_F} \tBb \times N \tvpb \right) \cdot N \tilde{\bn} = 
  \left(  N \tvpb \times  N \tilde{\bn} \right) \cdot \frac{DF}{J_F} \tBb.$$ 
Moreover, we insert the notation from Definition \ref{def:derivativevectors}
for the columns of the inverse transposed Jacobian matrix to represent the vector operations as
\begin{align*}
\int_{\partial \tilde \Omega} \left((
\bn_1 \tvp_1+ \bn_2 \tvp_2+  \bn_3 \tvp_3)
 \times \bn_1\right)\cdot DF \tBb \ \frac{\|\bt_2 \times \bt_3\|}{ J_F \|\bn_1\| }\du\tilde{\sigma}.
\end{align*}
Then, we make use of the vector identities from Proposition \ref{prop:bv} to obtain
\begin{align}\label{boundarypartampere}
\int_{\partial \tilde \Omega} \left(-\tvp_2 \frac{\bt_3}{J_F} +\tvp_3 \frac{\bt_2}{J_F}\right) \cdot DF \tBb \du\tilde{\sigma}.
\end{align}

Using the same formulas for the boundary part of Gauss' law gives us
\begin{align}\label{boundarypartgauss}
\int_{\partial\tilde \Omega} \tilde \psi  \left(N \tEb \cdot \frac{\bn_1}{\|\bn_1\|}\right)  \| \bt_2 \times \bt_3\| \du\tilde{\sigma}
=\int_{\partial\tilde \Omega} \tilde \psi  \left(N \tEb \right)\cdot \bn_1   |J_F| \du\tilde{\sigma}.
\end{align}

\end{proof}

\section{Spatial Semi-discretisation}
\label{sec:spatial_discretiastion}

In this section, we devise a spatial semi-discretisation based on conforming spline finite elements. For this purpose, we extend the curvilinear description proposed in \cite{perse2021geometric} for periodic boundary conditions to a toroidal coordinate system.

\subsection{Conforming Spline Finite Elements with Boundary Conditions}\label{sec:boundarysplines}
In this subsection, we review the construction of spline basis functions with real boundary conditions.
First, we review general properties of the basis splines given in \cite{de1972calculating}, which we will use further on.

Let us start with the knot vector $T=\{t_j\}_{1-p\leq j\leq N+p+1}$, which is a non-decreasing sequence of points. In our case, we have chosen the equidistant grid points $\xi_j$ of our $N_x$ cells for the knot sequence so that $\xi_{j+1}-\xi_j=\Delta \xi$.

From the knot sequence the $N$ splines of degree $p$ are defined according to the following formula:
\begin{mydef}
The $j$-th basis spline is computed via the recursion formula
\begin{align}\label{splinerecursion}
S^p_j(\xi)= \frac{\xi-t_j}{t_{j+p}-t_j} S^{p-1}_j(\xi) + \frac{t_{j+p+1} - \xi }{t_{j+p+1}-t_{j+1}} S^{p-1}_{j+1}(x),
\end{align}
where the spline of degree zero is defined as $S^0_j(x) = \chi_{[t_j,t_{j+1}]}$.
Furthermore, the derivative of the $j$-th spline is calculated as
\begin{align}\label{splinederivative}
\frac{d S_j^p(\xi)}{\du\xi}=p\left(\frac{S_j^{p-1}(\xi)}{t_{j+p}-t_j}-\frac{S_{j+1}^{p-1}(\xi)}{t_{j+p+1}-t_{j+1}}\right).
\end{align} 
\end{mydef}
In \cite{perse2021geometric}, we have worked with basis splines that are defined on a periodic knot sequence, which has the following form:
\begin{align*}
T=\{\xi_{N-p+1},...,\xi_N,\xi_1,\xi_2,...,\xi_{N-1},\xi_N,\xi_1,...,\xi_{p+1} \}.
\end{align*}
Here, the spline values are collected in a vector as $\Sb^p(\xi)=(S^p_{1}(\xi),..., S^p_N(\xi))$.
 
Now, for non-periodic boundary conditions, we consider clamped splines. In this case, the outer grid points are repeated $p$ times so that they have multiplicity $p+1$. Note that without assuming periodicity  we have $N+1$ grid points for $N$ cells. Then, the knot sequence is given by
\begin{align*}
T=\{\xi_1,...,\xi_1,\xi_2,...,\xi_N,\xi_{N+1},...,\xi_{N+1}\}.
\end{align*}
From this knot sequence the $N+p$ splines of degree $p$ are defined again via the recursion formula \eqref{splinerecursion}, where for dimensionality reasons we need to consider an additional first and last zero spline of degree $p-1$. 

In our convention, we denote the spline starting in the first cell as $S_1^p$. Accordingly, the first spline and last spline of degree $p$ are computed via \eqref{splinerecursion} as
\begin{align*}
S^p_{1-p}(\xi)=& \frac{t_{2} - \xi }{t_{2}-t_{2-p}} S^{p-1}_{2-p}(\xi)=  \frac{\xi_2 - \xi }{\Delta \xi} S^{p-1}_{2-p}(\xi), \\
S^p_{N}(\xi)=& \frac{\xi-t_{N}}{t_{N+p}-t_{N}} S^{p-1}_{N}(\xi)=\frac{\xi-\xi_{N}}{\Delta \xi} S^{p-1}_{N}(\xi). 
\end{align*}

Since the other splines equal zero at the boundary, we obtain
\begin{align*}
 S^p_j(0)=\begin{cases} & 1 \text{ when } j=1-p, \\ 
& 0 \text{ else,}
\end{cases}
\quad 
S^p_j(1)=\begin{cases}&1 \text{ when } j=N,\\ 
& 0 \text{ else.}
\end{cases}
 \end{align*}
 
This leads to the following evaluation of the product of two splines at the boundary:

\begin{align*}
 [S_i^{p-1} S_j^p]_0^1&=S_i^{p-1}(1) S^p_j(1) -S_i^{p-1}(0) S^p_j(0)
 \\
 &= \begin{cases} -1 & \text{ when } i=1-(p-1) \land  j=1-p, \\
1 & \text{ when } i=j=N,\\
0 & \text{ else. } \end{cases}
\end{align*}

We collect the spline values in a vector as $\Sb_\star^p(\xi)=(S^p_{1-p}(\xi),..., S^p_N(\xi))$ and write the spline derivative in matrix vector form with the help of the discrete 1D derivative matrix $D_\star$ defined via $ \frac{\du}{\du\xi} \Sb_\star^p(\xi) = \Sb_\star^{p-1}(\xi) D_\star$. The entries of the matrix are computed using the formula for the spline derivative \eqref{splinederivative},
\begin{align*}
\frac{\du S_{1-p}^p(\xi)}{\du\xi} &=-\frac{p}{\Delta \xi}S_{2-p}^{p-1}(\xi), 
\\
\frac{\du S_{j-p}^p(\xi)}{\du\xi} &=\frac{p}{\Delta \xi} \left( \frac{S_{j-p}^{p-1}(\xi)}{j-1}-\frac{S_{j+1-p}^{p-1}(\xi)}{j}\right) \text{ for } 2 \leq j \leq p,
 \\
\frac{\du S_j^p(\xi)}{\du\xi} &=\frac{1}{\Delta \xi} \left( S_j^{p-1}(\xi)-S_{j+1}^{p-1}(\xi)\right) \text{ for } 1 \leq j \leq N-p,
 \\
\frac{\du S_{N-j}^p(\xi)}{\du\xi} &=\frac{p}{\Delta \xi} \left( \frac{S_{N-j}^{p-1}(\xi)}{j+1}-\frac{S_{N-j+1}^{p-1}(\xi)}{j}\right) \text{ for } 1 \leq j \leq p-1,
\\
 \frac{\du S_{N}^p(\xi)}{\du\xi} &=\frac{p}{\Delta \xi}S_{N}^{p-1}(\xi).
\end{align*}
Then, the matrix $D_\star\in\mathbb{R}^{(N+p)\times (N+p)}$ is given by 
\small
\begin{align*}
D_\star=\frac{1}{\Delta \xi} \begin{pmatrix}
0&0&0&0&0&0&0&0&0&0\\
-\frac{p}{1}&\frac{p}{1}&0&0&0&0&0&0&0&0\\
0&\ddots&\ddots&0&0&0&0&0&0&0\\
0&0&-\frac{p}{p-1}&\frac{p}{p-1}&0&0&0&0&0&0\\
0&0&0&-1&1&0&0&0&0&0\\
\vdots&\vdots&\vdots&\ddots&\ddots&\ddots&\ddots&\vdots&\vdots&\vdots\\
0&0&0&0&0&-1&1&0&0&0\\
0&0&0&0&0&0&-\frac{p}{p-1}&\frac{p}{p-1}&0&0\\
0&0&0&0&0&0&0&\ddots&\ddots&0\\
0&0&0&0&0&0&0&0&-\frac{p}{1}&\frac{p}{1}
\end{pmatrix},
\end{align*}
\normalsize
where the first row accounts for the one spline less we have with degree $p-1$.

The 3D spline basis for differential 0-forms is constructed as a tensor product of the 1D splines, 
\begin{align}\label{boundaryspline0}
\tLa^0(\xib)=\Sb_\star^p(\xi_1) \otimes \Sb^p(\xi_2) \otimes \Sb^p(\xi_3).
\end{align}

Since we assumed non-periodic boundary conditions only in the first direction,  
we take clamped splines in the first direction and periodic splines in the other two directions.
Then, we build the 3D derivative matrices accordingly,
\begin{align}\label{boundaryderivativematrix}
\G=\begin{pmatrix}
D_1\\
D_2\\
D_3
\end{pmatrix},
\C=\begin{pmatrix}
0&-D_3&D_2\\
D_3&0&-D_1\\
-D_2&D_1&0
\end{pmatrix},
\D = \G^\top, 
\end{align}
where the block matrices $D_{1,2,3}$ are constructed as the tensor product of 1D derivative matrices via
$D_1=D_\star\otimes \mathbb{I}\otimes \mathbb{I}, D_2=\mathbb{I}\otimes D \otimes \mathbb{I}, D_3=\mathbb{I}\otimes \mathbb{I}\otimes D $. Here, $\mathbb{I}$ stands for the identity matrix and the periodic derivative matrix is computed via $\frac{\du}{\du\xi}\Sb^p(\xi)=\Sb^{p-1}(\xi) D$,
$$D=\frac{1}{\Delta \xi} \begin{pmatrix}
1&0&...&0&-1\\
-1&1&0&...&0\\
0&-1&1&0&\vdots\\
\vdots&0&\ddots&\ddots&0\\
0&...&0&-1&1
\end{pmatrix}.$$

Then, the 3D spline basis functions for the differential 1-,2- and 3-forms are defined as
\begin{equation} \label{boundarysplines}
\begin{aligned}
\tLab^1(\xib)&=\begin{pmatrix}
\tLa^{1,1}(\xib)&0&0\\
0&\tLa^{1,2}(\xib)&0\\
0&0&\tLa^{1,3}(\xib)
\end{pmatrix}, \\
\tLab^2(\xib)&=\begin{pmatrix}
\tLa^{2,1}(\xib)&0&0\\
0&\tLa^{2,2}(\xib)&0\\
0&0&\tLa^{2,3}(\xib)
\end{pmatrix}, 
\\
\tLa^3(\xib)&=\Sb_\star^{p-1}(\xi_1) \otimes \Sb^{p-1}(\xi_2) \otimes \Sb^{p-1}(\xi_3).
\end{aligned}
\end{equation}
with
{\small
\begin{align*}
\tLa^{1,1}(\xib) &= \Sb_\star^{p-1}(\xi_1) \otimes \Sb^p(\xi_2) \otimes \Sb^p(\xi_3), & \tLa^{2,1}(\xib) &= \Sb_\star^p(\xi_1) \otimes \Sb^{p-1}(\xi_2) \otimes \Sb^{p-1}(\xi_3),
\\ \tLa^{1,2}(\xib) &= \Sb_\star^p(\xi_1) \otimes \Sb^{p-1}(\xi_2) \otimes \Sb^p(\xi_3), & \tLa^{2,2}(\xib) &= \Sb_\star^{p-1}(\xi_1) \otimes \Sb^p(\xi_2) \otimes \Sb^{p-1}(\xi_3), 
\\
 \tLa^{1,3}(\xib) &= \Sb^p_\star(\xi_1) \otimes \Sb^p(\xi_2) \otimes \Sb^{p-1}(\xi_3), &  \tLa^{2,3}(\xib) &= \Sb_\star^{p-1}(\xi_1) \otimes \Sb^{p-1}(\xi_2) \otimes \Sb^p(\xi_3).
\end{align*}}
\begin{prop}\label{prop:boundaryderham}
The spline basis functions $\tLa^0,\tLab^1,\tLab^2,\tLa^3$ defined in \eqref{boundaryspline0} and \eqref{boundarysplines} form a discrete de Rham sequence with the derivative matrices $\G,\C,\D$  \eqref{boundaryderivativematrix}, i.e., it holds that $\C\G = 0$, $\D\C = 0$ and
\begin{equation}
\begin{aligned}\label{boundaryderham}
\nabla_\xib \tLa^0 &= \tLab^1 \G, \\
\nabla_\xib \times \tLab^1 &= \tLab^2 \C, \\
\nabla_\xib \cdot \tLab^2 &= \tLa^3 \D.
\end{aligned}
\end{equation}
\end{prop}
\begin{proof}
Since we defined the derivative matrices to reproduce the partial derivatives of the splines on the level of the degrees of freedom, \eqref{boundaryderham} holds by construction. So, it is left to check that $\C \G=0$ and $\D \C = 0$. When computing $\C \G$ block-wise, we get
\begin{align*}
\C \G = \begin{pmatrix}
-D_3 D_2 + D_2 D_3 \\
D_3 D_1 - D_1 D_3\\
-D_2 D_1 + D_1 D_2
\end{pmatrix}.
\end{align*}
Then, the Kronecker product structure of the derivative matrices guarantees that the matrices commute because the identity matrix commutes with every matrix, e.g.~
{\small
\begin{align*}
D_3 D_1 =(\mathbb{I}\otimes \mathbb{I}\otimes D) (D_\star\otimes \mathbb{I}\otimes \mathbb{I}) = D_\star \otimes \mathbb{I}\otimes D= (D_\star\otimes \mathbb{I}\otimes \mathbb{I}) (\mathbb{I}\otimes \mathbb{I}\otimes D)= D_1 D_3.
\end{align*}}
Analogously, we can verify that $\D \C=0$.  
\end{proof}

\subsection{Compatible FEM Description of the Field Equations}

To discretise Maxwell's equations based on the compatible finite element spaces, we represent the electromagnetic fields with a finite number of degrees of freedom, $\teb\in \R^{ 3N_1 \times 1}, \tbb\in \R^{ 3 N_2 \times 1}$, as 
\begin{subequations}\label{fieldsplinerep}
\begin{align}\label{ef}
&\tEb_h(\xib,t)=\tLaB^1(\xib) \teb(t),
\\\label{bf} 
&\tBb_h(\xib,t)=\tLaB^2(\xib) \tbb(t).
\end{align}
\end{subequations}
The fields in physical space can be expressed in the following way based on the Piola transforms \eqref{PiolaE} and \eqref{PiolaB},
\begin{align*} 
\Eb_h(\xb,t)&= N(\xib) \tEb_h(\xib,t)=  N(\xib) \tLaB^1(\xib) \teb(t),
\\
\Bb_h(\xb,t)&= \frac{DF(\xib)}{J_F(\xib)} \tBb_h(\xib,t)=\frac{DF(\xib)}{J_F(\xib)} \tLaB^2(\xib) \tbb(t).
\end{align*}

Furthermore, the mass matrices for the differential forms are defined as
\begin{equation}
\begin{aligned} \label{mass}
(\tM_0)_{ij}&= \int_{{\tom}} \tLa^0_i(\xib) \tLa^0_j(\xib) |J_F(\xib)| \ \du\xib \text{ for } 1 \leq i,j \leq N_0, 
  \\
 (\tM_1)_{IJ}&= \int_{{\tom}} \tLab^1_I(\xib)^\top G_m^{-1}(\xib) \tLab^1_{J}(\xib) |J_F(\xib)| \ \du\xib \text{ for } 1 \leq I,J \leq 3 N_1, 
  \\
  (\tM_2)_{IJ}& = \int_{{\tom}} \tLab^2_I(\xib)^\top G_m(\xib) \tLab^2_J(\xib) \frac{1}{|J_F(\xib)|} \ \du\xib
 \text{ for } 1 \leq I,J \leq 3N_2,
 \\
(\tM_3)_{ij}&= \int_{{\tom}} \tLa^3_i(\xib)  \tLa^3_j(\xib)  \frac{1}{|J_F(\xib)|} \ \du\xib \text{ for } 1 \leq i,j \leq N_3. 
\end{aligned}
\end{equation}

Next, we introduce the boundary matrices. Therefore, we insert the spline representation of the magnetic field \eqref{bf} into the boundary part of Amp\`ere's law \eqref{boundarypartampere} and test with the respective spline basis function $\tvpb=\tLab^1$  to obtain
\begin{align*}
\int_{\partial \tilde \Omega} \left(0, -\tLa^1_2 \frac{\bt_3}{J_F} ,\tLa^1_3 \frac{\bt_2}{J_F}\right) \cdot (\bt_1\tLa^2_1,\bt_2\tLa^2_2,\bt_3\tLa^2_3) \du\tilde{\sigma} \tbb.
\end{align*}
Then, we define the 1-form boundary matrix as 
{\footnotesize
\begin{align*}
\tM^1_b&= \int_0^1 \int_0^1 \left(  \frac{1}{J_F}\begin{pmatrix}
0 & 0& 0 \\
-\bt_3 \cdot \bt_1 \tLa^{1,2} \tLa^{2,1}  & -\bt_3 \cdot \bt_2 \tLa^{1,2} \tLa^{2,2}  &-\bt_3 \cdot \bt_3 \tLa^{1,2} \tLa^{2,3}  \\
\bt_2 \cdot \bt_1 \tLa^{1,3} \tLa^{2,1} & \bt_2 \cdot \bt_2 \tLa^{1,3} \tLa^{2,2} & \bt_2 \cdot \bt_3 \tLa^{1,3} \tLa^{2,3}
\end{pmatrix} \right)\bigg|_{\xi_1=0}^1  \du\xi_2 \du\xi_3.
\end{align*}}

Analogously, we use the spline representation of the electric field, \eqref{ef}, to rewrite the boundary part of Gauss' law \eqref{boundarypartgauss} and test with the respective spline basis function $\tilde \psi=\tLa^0$,
\begin{align*}
\int_{\partial\tilde \Omega} \tLa^0  \left((\bn_1 \tLa^1_1, \bn_2\tLa^1_2,  \bn_3\tLa^1_3)\cdot \bn_1\right)  |J_F| \du\tilde{\sigma} \teb.
\end{align*}
Then, the 0-form boundary matrix is defined as
\begin{align*}
\tM_b^0 =\int_0^1\int_0^1 \left( \tLa^0 \left(\bn_1 \cdot \bn_1 \tLa^1_1,\bn_1 \cdot \bn_2 \tLa^1_2, \bn_1 \cdot \bn_3 \tLa^1_3\right)  |J_F| \right)\bigg|_{\xi_1=0}^1 \du\xi_2\du\xi_3.
\end{align*}

The transformed discrete versions of Maxwell's equations from Proposition \ref{prop:transformedVM} take the following form in matrix notation:
\begin{subequations}
\begin{align}\label{discreteAmpere}
\tM_1 \dot{\teb}&= \C^\top \tM_2 \tbb + \tM^1_b \tbb -  \tjb,
\\ \label{discreteFaraday}
\dot{\tbb}&= -\C \teb,
\\ \label{discreteGauss}
\G^\top\tM_1 \teb&= \tM_b^0 \teb - \trhob ,
\\ \notag
\D \tbb&=0,
\end{align}
\end{subequations}
where the discrete current and charge can be written as $\tjb=\WM_q  \tLaBB^1(\Xib)^\top \NN^\top(\Xib)\Vb$, $\trhob=\WM_q \tLaBB^0(\Xib)^\top \mathbbm{1}_{N_p}$ with the particle weight matrices $\WM_q:=\operatorname{diag}(\omega_p q_p)\otimes \mathbb{I}_3,$  $\WM_m:=\operatorname{diag}(\omega_p m_p)\otimes \mathbb{I}_3$.

For this system, the energy balance is described by the following theorem:
\begin{theorem}
The energy balance of the semi-discrete field energy, $\tilde{\mathcal{H}}_{EB}$, is given by
\begin{align} \label{fieldenergy}
\frac{\du \tilde{\mathcal{H}}_{EB}}{\du t} = \teb^\top  \tM^1_b \tbb - \teb^\top \tjb. 
\end{align}
\end{theorem}
\begin{proof}
The semi-discrete field energy is defined as $\tilde{\mathcal{H}}_{EB}= \frac{1}{2} \teb^\top \tM_1 \teb + \frac{1}{2} \tbb^\top \tM_2 \tbb $
 and so the time derivative of this term gives us
\begin{align*}
\frac{\du \tilde{\mathcal{H}}_{EB}}{\du t} = \teb^\top \tM_1 \dot \teb +  \tbb^\top \tM_2 \dot \tbb.
\end{align*} 
We insert the discrete forms of Amp{\`e}re's \eqref{discreteAmpere} and Faraday's laws \eqref{discreteFaraday} to see that
\begin{align*}
\teb^\top \tM_1 \dot \teb +  \tbb^\top  \tM_2 \dot \tbb 
= & \teb^\top  \C^\top \tM_2 \tbb + \teb^\top  \tM^1_b \tbb -  \teb^\top  \WM_q  \tLaBB^1(\Xib)^\top \NN^\top(\Xib)\Vb -\tbb^\top  \tM_2 \C \teb
\\ 
= & \teb^\top  \tM^1_b \tbb - \teb^\top \tjb.
\end{align*}
\end{proof}

Here, the semi-discrete Poynting flux, which gives the field energy crossing the boundary, is computed as
\begin{align}\label{Poyntingflux}
\int_{\partial \Omega} (\Eb_h \times \Bb_h) \cdot \bn \du\sigma = -\teb^\top  \tM^1_b \tbb. 
\end{align}

\subsection{Perfect Conductor Boundary Conditions}\label{sec:perfectconductorboundary}
For simulations all the way to the wall of the fusion device, perfect conductor boundary conditions as described in \cite{guo1993global} are physically relevant, 
\begin{align}\label{essential}
\tEb \times \tilde \bn = 0, 
\end{align} 
This also implies that with matching initial conditions
\begin{align}\label{derived}
\dot{\tBb} \cdot \tilde \bn=0.
\end{align}  
This can be seen by taking the scalar product of Faraday's law \eqref{strongFaraday} with the normal vector, 
\begin{align*}
\frac{\partial}{\partial t}\tBb \cdot \bn &= -(\nabla_\xib \times \tEb) \cdot \tilde \bn = 0,
\end{align*}
since $\tilde{\bn}$ is constant.
Hence, we have found our pair of boundary conditions for the electromagnetic fields. 

In our case, the normal vector simplifies to $\tilde{\bn}=(\pm 1,0,0)^\top$ so that the perfect conductor boundary conditions translate to $$(\tEb \times (\pm1,0,0)^\top)\bigg|_{\partial \tom} = \pm\left((0,\tE_3(\xib), -\tE_2(\xib))^\top \right)\bigg|_{\partial \tom} = \mathbf{0}$$ and  $$\left(\frac{\partial \tBb}{\partial t} \cdot (\pm1,0,0)^\top \right) \bigg|_{\partial \tom} = \pm\dot{\tB}_1(\xib)\bigg|_{\partial \tom}= 0.$$

To satisfy the boundary conditions \eqref{essential}, the differential 1-forms have to be in the following constrained Sobolev space: 
$$H_0(\curl,\tom):= \{\boldsymbol{\omega} \in L^2(\tom)^3 | \curl \boldsymbol{\omega} \in L^2(\tom)^3 \land \left(\boldsymbol{\omega}\times \tilde{\bn}\right)\bigg|_{\partial \tom}=\mathbf{0} \}.$$
Additionally, to satisfy the derived boundary conditions \eqref{derived}, the differential 2-forms have to be in 
$$ H_0(\div,\tom):= \{\boldsymbol{\omega} \in L^2(\tom)^3 |  \div \boldsymbol{\omega} \in L^2 (\tom) \land \left(\boldsymbol{\omega} \cdot \tilde{\bn}\right)\bigg|_{\partial \tom}=0 \}.$$

For our discretisation with spline finite elements, we impose \eqref{essential} as essential boundary conditions. This yields the following constraints on the spline basis functions:  
\begin{equation}\label{perfectconductorsplines}
\begin{aligned}
\tLa^{1,2}(\xib)\bigg|_{\partial \tom} =0,
\quad
\tLa^{1,3}(\xib)\bigg|_{\partial \tom} =0,
\quad
\tLa^{2,1}(\xib)\bigg|_{\partial \tom} =0.
\end{aligned}
\end{equation}
Note that the constrained spline finite element spaces resulting from imposed Dirichlet boundary conditions are given in \cite[Sec.~3.3]{BUFFA20101143}.

Consequently, the boundary term from Amp{\`e}re's equation vanishes for the perfect conductor boundary conditions, since the 1-form boundary matrix $\tM_b^1$ equals zero:
{\footnotesize
\begin{align} \label{perfectconductorboundarymatrix}
\int_0^1 \int_0^1 \left(  \frac{1}{J_F}\begin{pmatrix}
0 & 0& 0 \\
-\bt_3 \cdot \bt_1 \tLa^{1,2} \tLa^{2,1}  & -\bt_3 \cdot \bt_2 \tLa^{1,2} \tLa^{2,2}  &-\bt_3 \cdot \bt_3 \tLa^{1,2} \tLa^{2,3}  \\
\bt_2 \cdot \bt_1 \tLa^{1,3} \tLa^{2,1} & \bt_2 \cdot \bt_2 \tLa^{1,3} \tLa^{2,2} & \bt_2 \cdot \bt_3 \tLa^{1,3} \tLa^{2,3}
\end{pmatrix} \right)\bigg|_{\partial \tom}  \du\xi_2 \du\xi_3 = \mathbf{0}.
\end{align} }
This means that the Poynting flux \eqref{Poyntingflux} equals zero, too. Therefore, we do not have field energy exchange over the boundary and the system is closed.

In Proposition \ref{prop:boundaryderham}, we have proven that the clamped basis splines by construction satisfy a discrete de Rham sequence. Now, we assume perfect conductor boundary conditions \eqref{perfectconductorsplines} for the spline basis functions and have to check that \eqref{boundaryderham} still holds at the boundary. 
A compatibility condition is given by the following proposition:
\begin{prop}
Under the assumption of perfect conductor boundary conditions, 
the clamped splines constructed in subsection \ref{sec:boundarysplines} form a discrete de Rham sequence with boundary if and only if they satisfy the condition
\begin{align}\label{S_starzero}
\Sb_\star^p(\xi_1)\bigg|_{\partial \tom} = 0.
\end{align}
\end{prop}
\begin{proof}
We want to show that the B-splines still form a discrete de Rham sequence, i.e. the three properties in \eqref{boundaryderham} are satisfied. We start with the first property stating
\begin{align*}
\nabla_\xib \tLa^0 = \tLab^1 \G.
\end{align*}

Then, for a 0-form $\tPhi=\tLa^0(\xi) \tphib$,  we obtain at the boundary
\begin{align*}
\left(\nabla_\xib \tLa^0(\xib) \right)\bigg|_{\partial \tom} \tphib=&
\begin{pmatrix}
\partial_{\xi_1} \Sb_\star^p(\xi_1)\bigg|_{\partial \tom} \otimes \Sb^p(\xi_2)  \otimes \Sb^p(\xi_3) \\
 \Sb_\star^p(\xi_1)\bigg|_{\partial \tom} \otimes \partial_{\xi_2} \Sb^p(\xi_2)  \otimes \Sb^p(\xi_3) \\
 \Sb_\star^p(\xi_1)\bigg|_{\partial \tom} \otimes \Sb^p(\xi_2)  \otimes \partial_{\xi_3} \Sb^p(\xi_3) 
\end{pmatrix} \tphib \\
\stackrel{!}{=} &
\begin{pmatrix}
\tLa^{1,1}(\xib)\bigg|_{\partial \tom} D_1 \phib\\0\\0
\end{pmatrix} \stackrel{\eqref{perfectconductorsplines}}{=} \tLab^1(\xib)\bigg|_{\partial \tom} \G \tphib.
\end{align*}
This is satisfied non-trivially if and only if $\Sb_\star^p(\xi_1)\bigg|_{\partial \tom} = 0$.

Next, for the second property of \eqref{boundaryderham}, we look at the rotation of a 1-form $\tAb=\tLab^1 \tab$ at the boundary,
{\footnotesize
\begin{align*}
&(\nabla_\xib \times \tLab^1)\bigg|_{\partial \tom} \tab \\
&= \begin{pmatrix}
 \Sb_\star^p(\xi_1)\bigg|_{\partial \tom} \otimes \partial_{\xi_2} \Sb^p(\xi_2)  \otimes \Sb^{p-1}(\xi_3) -  \Sb_\star^p(\xi_1)\bigg|_{\partial \tom} \otimes \Sb^{p-1}(\xi_2)  \otimes \partial_{\xi_3} \Sb^p(\xi_3) 
 \\
 \Sb_\star^{p-1}(\xi_1)\bigg|_{\partial \tom} \otimes \Sb^p(\xi_2)  \otimes \partial_{\xi_3} \Sb^p(\xi_3)-
\partial_{\xi_1} \Sb_\star^p(\xi_1)\bigg|_{\partial \tom} \otimes \Sb^p(\xi_2)  \otimes \Sb^{p-1}(\xi_3) 
\\
\partial_{\xi_1} \Sb_\star^p(\xi_1)\bigg|_{\partial \tom} \otimes \Sb^{p-1}(\xi_2)  \otimes \Sb^p(\xi_3) - \Sb_\star^{p-1}(\xi_1)\bigg|_{\partial \tom} \otimes \partial_{\xi_2} \Sb^p(\xi_2)  \otimes \Sb^p(\xi_3) 
\end{pmatrix} \tab
\\
&\stackrel{!}{=} \begin{pmatrix}
0\\
\tLa^{2,2}(\xib)\bigg|_{\partial \tom} (D_3 \tab_1 - D_1 \tab_3) \\
\tLa^{2,3}(\xib)\bigg|_{\partial \tom} (D_1 \tab_2 - D_2 \tab_1)
\end{pmatrix} \stackrel{\eqref{perfectconductorsplines}}{=}  \tLab^2(\xib)\bigg|_{\partial \tom} \C \tab.
\end{align*}
}
This leads again to the compatibility condition $\Sb_\star^p(\xi_1)\bigg|_{\partial \tom} = 0$. 

Since there are no boundary conditions on the differential 3-form basis the third property of \eqref{boundaryderham} is also satisfied at the boundary,
\begin{align*}
\left(\nabla_\xib \cdot \tLab^2(\xib)\right)\bigg|_{\partial \tom} = \tLa^3(\xib)\bigg|_{\partial \tom} \D,
\end{align*}
which concludes the proof.
\end{proof}

From this proposition, it follows that also the boundary part from Gauss' law vanishes, since the 0-form boundary matrix equals zero,
\begin{align*}
\tM_b^0 = \int_0^1\int_0^1 \left( \tLa^0 \left(\bn_1 \cdot \bn_1 \tLa^{1,1},\bn_1 \cdot \bn_2 \tLa^{1,2}, \bn_1 \cdot \bn_3 \tLa^{1,3}\right)  J_F \right)\bigg|_{\partial \tom} \du\xi_2\du\xi_3 =0.
\end{align*}

We already showed that the initial boundary condition $\tBb \cdot \tilde \bn = C$ for the magnetic field is conserved over time because the magnetic field is updated with the curl of the electric field.  Assuming that the basis functions form a de Rham sequence and that  $\tBb$ is computed from the 1-form potential $\tAb$, we have $\tBb \cdot \tilde \bn=0$, 
since
{\footnotesize
\begin{align*}
\left(\tBb (\xib)\cdot \tilde{\bn}\right)\bigg|_{\partial \tom} =&\left(\nabla_\xib \times \tLab^1 \cdot \tilde{\bn} \right)\bigg|_{\partial \tom}  \tab \\
=& \Sb^p_\star(\xi_1)\bigg|_{\partial \tom} \otimes \partial_{\xi_2} \Sb^p(\xi_2) \otimes \Sb^p(\xi_3)  -  \Sb^p_\star(\xi_1)\bigg|_{\partial \tom} \otimes \Sb^p(\xi_2) \otimes  \partial_{\xi_3} \Sb^p(\xi_3) \tab = 0.
\end{align*} 
}
\subsection{Particle Boundary Conditions} \label{sec:particleboundary}
When the particle trajectories cross the boundary, we also need to impose some boundary conditions on them. 
Note that we avoid a possible singularity by excluding the pole on the physical mesh, as can be seen exemplarily in Figure \ref{fig:radialmeshes}. Therefore, on the logical mesh, we have an inner boundary at $\xi_1=0$ and an outer boundary at $\xi_1=1$.
Following \cite{chacon2019}, we consider reflecting boundaries. We make use of the normal vector $\bn=N(\xib) \tilde \bn=\bn_1$ to compute the reflection at the boundary as
\begin{itemize}
\item Inner boundary at $\xi_1=0$: $$\xi_1 = - \xi_1, \quad \vb = \vb -2(\bn\cdot \vb) \frac{\bn}{\|\bn\|^2}= \vb -2(\bn_1\cdot \vb) \frac{\bn_1}{\|\bn_1\|^2},$$
\item Outer boundary at $\xi_1=1$: 
$$\xi_1 = 2 - \xi_1, \quad \vb = \vb -2(\bn_1\cdot \vb) \frac{\bn_1}{\|\bn_1\|^2}.$$ 
\end{itemize}
The particle weight is kept constant. 
The reflecting boundary prevents heat fluxes and currents at the boundary and ensures exact energy conservation.

In this study, we additionally consider a constant in- and outflow of particles: A particle of identical weight is reinserted at the opposite boundary with the same velocity, which can be considered as periodic particle boundary conditions,
\begin{itemize}
\item Inner boundary at $\xi_1=0$: $\xi_1 = \xi_1+1, \vb = \vb,$
\item Outer boundary at $\xi_1=1$:  $\xi_1 = \xi_1-1, \vb = \vb.$
\end{itemize}
This choice again conserves mass, energy and magnetic momentum. This second boundary conditions mimics a periodic behaviour and is considered here as an intermediate step for verification purposes rather than being physically motivated.

\subsection{Semi-discrete Poisson System for Perfect Conductor Boundary Conditions}
\begin{prop}
The time evolution of the equations of motion for $\tub=(\Xib,\Vb,\teb,\tbb)$ with perfect conductor boundary conditions can be expressed as the Poisson system
\begin{align*}
\frac{\du \tub}{\du t} = \JJ(\tub) D\tilde{\mathcal{H}}_h,
\end{align*}
with the Poisson matrix
{\small
\begin{align}\label{Poissonmatrix}
\JJ=
 \begin{pmatrix}
0& \NN^\top(\Xib) \WM_m^{-1} &0&0\\
-\WM_m^{-1} \NN(\Xib) & \WM_{\frac{q}{m}} \NN(\Xib) \tBB(\Xib,\tbb) \NN^\top(\Xib) \WM_m^{-1} & \WM_{\frac{q}{m}} \NN(\Xib)\tLaBB^1(\Xib) \tM_1^{-1}&0\\
0&-\tM_1^{-1} \tLaBB^1(\Xib)^\top \NN^\top(\Xib) \WM_{\frac{q}{m}}  &0&\tM_1^{-1}\C^\top \\
0&0&-\C\tM_1^{-1}&0
\end{pmatrix}
\end{align}}
where $\tBB(\Xib,\tbb)$ is a $3 N_p \times 3 N_p$ block matrix with generic block 
\begin{align*}
\hat{\tilde B}_h(\xib_p,t)=\sum_{i=1}^{N_2} \begin{pmatrix}
0&\tb_{i,3}(t)\tLa_i^{2,3}(\xib_p)&-\tb_{i,2}(t)\tLa_i^{2,2}(\xib_p)\\
-\tb_{i,3}(t)\tLa_{i}^{2,3}(\xib_p)&0&\tb_{i,1}(t)\tLa_i^{2,1}(\xib_p)\\ \tb_{i,2}(t)\tLa_i^{2,2}(\xib_p)&-\tb_{i,1}(t)\tLa_i^{2,1}(\xib_p)&0
\end{pmatrix}.
 \end{align*}
\end{prop}
\begin{proof}
The Jacobi identity can be proven in the same way as for the case with periodic boundaries in \cite[Theorem 4.3]{perse2021geometric}. The important properties that the discretisation builds a discrete de Rham sequence, the antisymmetry of $\C$, the symmetry of the mass matrices and the diagonal structure of the particle weight matrices is still satisfied for clamped splines with perfect conductor boundary conditions.
\end{proof}

The discrete Hamiltonian takes the following form:
\begin{align}\label{Hamiltonian}
\tilde{\mathcal{H}}_h= \frac{1}{2}\Vb^\top \WM_m \Vb+\frac{1}{2}\teb^\top \tM_1 \teb+\frac{1}{2}\tbb^\top \tM_2 \tbb
\end{align}
and the derivative can be computed as
\begin{align*}
D\tilde{\mathcal{H}}_h(\tub)=(0,\WM_m \Vb , \tM_1 \teb,\tM_2 \tbb)^\top.
\end{align*}

Consequently, the equations of motion are 
\begin{equation}
\begin{aligned} \label{equations of motion}
\dot{\Xib}&=\NN^\top(\Xib)\Vb,
\\
\dot{\Vb}&=\WM_{\frac{q}{m}}\NN(\Xib) \left( \tLaBB^1(\Xib) \teb+ \tBB(\Xib, \tbb) \NN^\top(\Xib) \Vb \right),\\
\tM_1 \dot{\teb} &=\C^\top\tM_2 \tbb -  \tLaBB^1(\Xib)^\top \NN^\top(\Xib) \WM_q  \Vb,\\
\dot{\tbb}&=-\C \teb.
\end{aligned}
\end{equation}

This semi-discretisation with perfect conductor boundary conditions yields energy conservation due to the antisymmetry of the Poisson matrix. This can also be computed directly via
\begin{align*}
\frac{\du \tilde{\mathcal{H}}_h}{\du t }=
 &\Vb^\top \WM_m \dot{\Vb} + \frac{\du}{\du t }\tilde{\mathcal{H}}_{EB}
\\ 
\stackrel{ \eqref{equations of motion}, \eqref{fieldenergy}}{=} & \Vb^\top \WM_{q}\NN \left( \tLaBB^1 \teb+ (\NN^\top \Vb) \times \tBb \right) +   \teb^\top  \tM^1_b \tbb - \teb^\top \tjb
\\
\stackrel{ \eqref{perfectconductorboundarymatrix} }{=}&0.
\end{align*}

\section{Temporal Discretisation}
\label{sec:temporal_discretisation}
Following our earlier work \cite{perse2021geometric} on a periodic domain, we consider both a semi-explicit Poisson integrator based on a Hamiltonian splitting \cite{crouseilles2015hamiltonian,he2015hamiltonian} and an energy-conserving discrete gradient scheme \cite{quispel1996discrete,kormann2021energy} combined with an antisymmetric splitting of the Poisson matrix.

\subsection{Charge Conserving Time Discretisation}
We split the discrete Hamiltonian \eqref{Hamiltonian} into three parts,
\begin{align*} \tilde{\mathcal{H}}_h=\tilde{\mathcal{H}}_p+\tilde{\mathcal{H}}_E+\tilde{\mathcal{H}}_B
\end{align*}
with 
\begin{align*}
\tilde{\mathcal{H}}_p= \frac{1}{2}\Vb^\top \WM_m \Vb,\ \tilde{\mathcal{H}}_E=\frac{1}{2}\teb^\top \tM_1 \teb, \ \tilde{\mathcal{H}}_B=\frac{1}{2}\tbb^\top \tM_2 \tbb
\end{align*}
to obtain the three subsystems
\begin{align*}
\dot{\tub}=\{\tub,\tilde{\mathcal{H}}_p \},\  \dot{\tub}=\{\tub,\tilde{\mathcal{H}}_E \}, \ \dot{\tub}=\{\tub,\tilde{\mathcal{H}}_B \}.
\end{align*}

For $\tilde{\mathcal{H}}_E$, the discrete equations of motion are
\begin{align*}
\Vb^{n+1}&=\Vb^n+\Delta t\WM_{\frac{q}{m}} \NN (\Xib^n) \tLaBB^1(\Xib^n) \teb^n,
\\
\tbb^{n+1}&=\tbb^n-\Delta t \C \teb^n.
\end{align*}

For $\tilde{\mathcal{H}}_B$, we get 
\begin{align*}
\tM_1 \teb^{n+1}&=\tM_1 \teb^n+\Delta t \C^\top \tM_2 \tbb^n.
\end{align*}

For $\tilde{\mathcal{H}}_p$, we obtain the following equations: 
\begin{subequations}\label{hs}
\begin{align}
\Xib^{n+1}=&\Xib^n+ \Delta t \NN^\top\left(\overline{\Xib}\right)\overline{\Vb},
\\\notag
\Vb^{n+1}=&\Vb^{n}+ \Delta t \WM_{\frac{q}{m}} \NN\left(\overline{\Xib}\right) \tBB\left(\overline{\Xib},\tbb^n\right) \NN^\top\left(\overline{\Xib}\right)\overline{\Vb},
\\
\tM_1 \teb^{n+1}= &\tM_1 \teb^n-\int^{t^{n+1}}_{t^n}   \tLaBB^1(\Xib(\tau))^\top \du\tau \WM_q \NN^\top\left(\overline{\Xib}\right)\overline{\Vb},
\end{align}
\end{subequations}
where $\overline{\Xib}=\frac{\Xib^{n+1}+\Xib^n}{2}, \overline{\Vb}=\frac{\Vb^{n+1}+\Vb^n}{2}$ and $\Xib(\tau)= \frac{(t^{n+1}-\tau)\Xib^n+(\tau-t^n)\Xib^{n+1}}{\Delta t}$.
For the simulation results of this Hamiltonian splitting we use the acronym HS.

Based on the bracket in \cite{crouseilles2015hamiltonian}, which yields a "Pseudo-Poisson matrix" that does not satisfy the Jacobi identity, we get slightly different subsystems.  

For $\tilde{\mathcal{H}}_E$, the discrete equations of motion are given by
\begin{align*}
\Vb^{n+1}&=\Vb^n+\Delta t\WM_{\frac{q}{m}} \NN (\Xib^n) \tLaBB^1(\Xib^n) \teb^n,
\\
\tbb^{n+1}&=\tbb^n-\Delta t \C \teb^n.
\end{align*}

For $\tilde{\mathcal{H}}_B$, we get 
\begin{align*}
\left(\mathbb{I}-\frac{\Delta t}{2} \WM_{\frac{q}{m}} \NN\tBB(\Xib^n,\tbb^n)\NN^\top\right) \Vb^{n+1}&= \left(\mathbb{I}+\frac{\Delta t}{2} \WM_{\frac{q}{m}} \NN \tBB(\Xib^n,\tbb^n)\NN^\top\right) \Vb^n,
\\
\tM_1 \teb^{n+1}&=\tM_1 \teb^n+\Delta t \C^\top \tM_2 \tbb^n.
\end{align*}

For $\tilde{\mathcal{H}}_p$, we obtain the following equations: 
\begin{subequations}\label{cef}
\begin{align}
\Xib^{n+1}&=\Xib^n+  \Delta t \NN^\top\left(\bar{\Xib}\right) \Vb^n,
\\ 
\tM_1 \teb^n&= \tM_1 \teb^n- \int_{t^n}^{t^{n+1}} \tLaBB^1(\Xib(\tau))^\top  \du\tau \WM_q  \NN^\top\left(\bar{\Xib}\right) \Vb^n.
\end{align}
\end{subequations}
For the simulation results of this method, we use the acronym CEF.

In both semi-explicit time discretisation schemes, the charge conservation depends on the exact solution of the particle trajectory in the following part:
\begin{align*}
\tM_1 \frac{\teb^{n+1}-\teb^{n}}{\Delta t} &=- \int_{t^n}^{t^{n+1}} \tLaBB^1(\Xib(\tau))^\top  \du\tau \WM_q  \frac{\Xib^{n+1}-\Xib^{n}}{\Delta t}.
\end{align*}
Since the particle update is computed iteratively, it is important that we impose the boundary conditions after computing the midpoint $\bar{\Xib}=\frac{\Xib^{n+1}+\Xib^n}{2}$ in \eqref{hs} and \eqref{cef}.
In the case of a particle trajectory crossing the boundary, we solve this part with a split line integral. First, we compute the point of intersection of the particle trajectory with the boundary and compute the line integral up to that point. 
Second, we reflect the particle position and velocity at this point and compute the line integral between the point of intersection and the new particle position.

\subsection{Energy Conserving Time Discretisation}
Energy-conserving time discretisation can be obtained by applying the trapezoidal rule to the gradient of the Hamiltonian together with an antisymmetric form of the Poisson matrix. In order to simplify the non-linearity in the system, we apply this strategy to subsystems obtained by an antisymmetric splitting of the Poisson matrix (cf.~\cite{kormann2021energy}).
We split the discrete Poisson matrix \eqref{Poissonmatrix}, keeping its skew-symmetry in each subsystem. So, we obtain the following four subsystems:
\begin{align*}
&\text{ system 1: } 
\dot{\Xib}= \NN^\top(\Xib)\Vb,
\\
&\text{ system 2: }
\dot{\Vb}= \WM_{\frac{q}{m}} \NN(\Xib) \tBB(\Xib,\tbb) \NN^\top(\Xib) \Vb,\\
&\text{ system 3: }
 \dot{\tbb}= -\C \teb, \ \tM_1 \dot{\teb}=\C^\top \tM_2 \tbb,
\\
&\text{ system 4: }
\dot{\Vb}= \WM_{\frac{q}{m}} \NN (\Xib) \tLaBB^1(\Xib) \teb, \
\tM_1 \dot{\teb} = - \tLaBB^1(\Xib)^\top \NN^\top(\Xib) \WM_q  \Vb.
\end{align*}

System 1 is discretised as
\begin{align*}
\Xib^{n+1} = \Xib^n+\Delta t \frac{\NN^\top(\Xib^{n+1})+ \NN^\top(\Xib^n)}{2} \Vb^n.
\end{align*}

For System 2, we get the following discretised equations:
\begin{align*}
\left(\mathbb{I}-\frac{\Delta t}{2} \WM_{\frac{q}{m}} \NN\tBB \NN^\top \right)\Vb^{n+1} =\left(\mathbb{I}+\frac{\Delta t}{2} \WM_{\frac{q}{m}} \NN \tBB \NN^\top\right) \Vb^{n}.
\end{align*}

System 3 is decoupled with the Schur complement $\mathsf{S}=\tM_1+\frac{\Delta t^2}{4} \C^\top \tM_2 \C$,
\begin{align*}
\teb^{n+1}&=\mathsf{S}^{-1}\left((\tM_1-\frac{\Delta t^2}{4} \C^\top \tM_2 \C)\teb^n+\Delta t \C^\top \tM_2 \tbb^n \right),\\
\tbb^{n+1}&=\tbb^n-\frac{\Delta t}{2} \C(\teb^{n+1}+\teb^n).
\end{align*}
Likewise, we obtain the decoupled equations for system 4
\begin{align*}
\teb^{n+1}&=\mathsf{S}^{-1}\left((\tM_1-\frac{\Delta t^2}{4} \WM_q \WM_{\frac{q}{m}} \mathsf{M}^\star)\teb^n-\Delta t (\tLaBB^1)^\top \NN^\top  \WM_q \Vb^n \right),\\
\Vb^{n+1}&=\Vb^n+\frac{\Delta t}{2} \WM_{\frac{q}{m}} \NN \tLaBB^1(\teb^{n+1}+\teb^n),
\end{align*}
where we introduced the particle mass matrix $\mathsf{M}^\star:=(\tLaBB^1)^\top \NN^\top \NN \tLaBB^1$ and the Schur complement $\mathsf{S}=\tM_1+\frac{\Delta t^2}{4}\WM_q \WM_{\frac{q}{m}}\mathsf{M}^\star.$ 
The simulation results of this energy conserving discrete gradient method are labeled as DisGradE.

The energy conservation of the implicit time discretisation method depends on the solution of the antisymmetric subsystems. Therefore, the DisGradE method stays energy conserving because the particle position is updated independently of the particle velocity and the electromagnetic fields provided that the number of particles is constant.

\section{Preconditioner for the Inversion of the Spline Mass Matrices}\label{sec:preconditioner}
Our methods require the inversion of the finite-element mass matrices based on $p-$th order splines. Since the condition number increases exponentially with the degree $p$, there is a need for preconditioning in an iterative solver. Donatelli et al.~\cite{donatelli2015robustgalerkin} have proposed a multilevel solver with a smoother based on a preconditioned conjugate gradient (PCG) solver. The smoothing is based on the eigenvalues of the matrices on a periodic uniform grid. The basic idea behind is that circulant matrices are diagonal in Fourier space and  hence, can be cheaply inverted based on the Fast Fourier Transform (FFT). On a periodic tensor product grid without coordinate transformations, the finite element matrices are circulant so that the linear equation systems can be solved directly after Fourier transformation yielding a very efficient solver compared to iterative solvers in this case. We refer to \cite{kormann2021energy} for a detailed description of this solution strategy. In this work, the matrices are no longer circulant for two reasons: Due to the non-periodic boundaries and due to the coordinate transformation. 

Observing the structure of the mass matrices for clamped splines of degree $p$, we notice that only the first and last $p+2$ rows differ from the mass matrix for periodic splines of the same degree. This observation motivates the alternative idea to use the eigenvalues of the periodic matrix only for the middle part that is identical to a periodic mass matrix, except for the periodicity at the boundary, and invert the boundary part, which consists of the first and last $p+2$ rows of the mass matrix, separately. 
\begin{figure}[!ht]
\centering
\begin{subfigure}{0.8\textwidth}
  \centering
  \includegraphics[width=\linewidth]{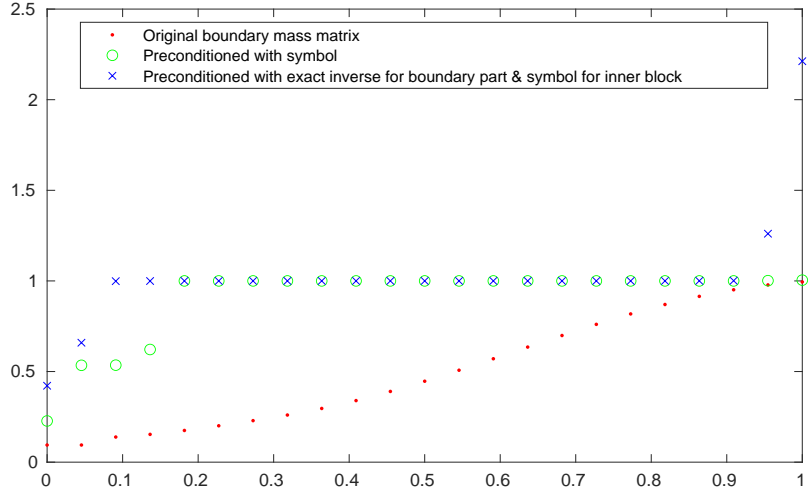}
\caption{Quadratic spline}
 \end{subfigure}
 \hfill
 \begin{subfigure}{0.8\textwidth}
  \centering
  \includegraphics[width=\linewidth]{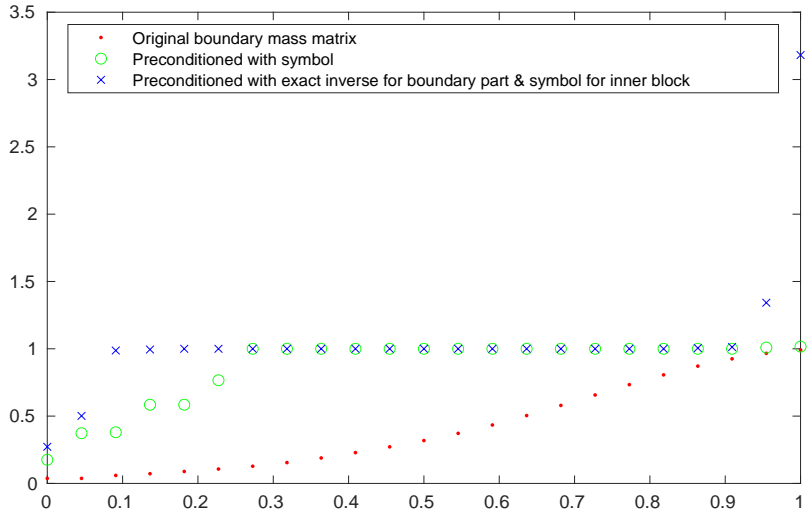}
\caption{Cubic spline}
 \end{subfigure}
\caption{Analytical eigenvalues of the 1D preconditioned mass matrix for different spline degrees. }
\label{fig:fftprec}
\end{figure}

Figure \ref{fig:fftprec} shows the eigenvalues of the 1D mass matrix with clamped splines for the two preconditioner ideas.  
We see that both methods improve the distribution of the eigenvalues of the preconditioned matrix. Since we are only interested in the ratio between the lowest and the highest eigenvalue, we decide to use the eigenvalue solver on the whole mass matrix, which has a lower condition number $\kappa=\frac{eigval_{max}}{eigval_{min}}$.

When it comes to the coordinate transformation, we have the effect that not all cells carry the same weight. The problem is especially highlighted for a mapping with singularity, which we need to lift as investigated in \cite[Sec.~3.2]{donatelli2015robustcollocation}.  Assuming for instance that the entries of $\tM$ are of minimal order $ \mathcal{O}\left(\frac{1}{\xi_1}\right)$, we want to precondition with a matrix which has entries of maximal order $ \mathcal{O}(\xi_1)$ and is close to an inverse of $\tM$.

The common choice is a Jacobi preconditioner based on the main diagonal matrix $(\mathsf{D}^2)_{ii} = \tM_{ii}$. In our case, we have used the diagonal row lumped mass $(\tM_{lump}^2)_{ii}=\sum_j \tM_{ij}$, which in many cases, yields better results. However, since the entries may become very small or even negative, we defined a positive lower limit for the sum of the row.

Finally, our preconditioner consists of the inverse of the lumped mass or the main diagonal matrix and the eigenvalue preconditioner $P_{fft}$, which is used on the system that is then almost uniform. In order not to destroy the symmetry of the matrix, we construct the preconditioner as
\begin{align*}
\mathsf{P}=\tM_{lump}^{-1} \mathsf{P}_{fft} \tM_{lump}^{-1} \text{  or  } \mathsf{P}=\mathsf{D}^{-1} \mathsf{P}_{fft} \mathsf{D}^{-1},
\end{align*} 
where the entries of the inverse matrices can easily be computed as $(\tM_{lump}^{-1})_{ii}= \frac{1}{\sqrt{\sum_j \tM_{ij}}}$ and $(\mathsf{D}^{-1})_{ii}=\frac{1}{\sqrt{\tM_{ii}}}$.

\begin{table}[htbp]
\centering
\caption{Number of iterations for the CG and PCG solver of the mass matrices for spline degree $p$. }
\begin{tabular}{|c||r|r|r|r|}
\hline
\multirow{3}{*} {Grid} &\multicolumn{4}{|c|}{ Number of iterations CG solver}  \\
\cline{2-5}
& \multicolumn{2}{|c|}{ $p_x=2$ } & \multicolumn{2}{|c|}{ $p_x=3$}  \\
\cline{2-5}
& $ N_x=8 $ & $N_x=32$ & $ N_x=8 $ & $N_x=32$ \\
\hline
Cartesian & $426$ & $451$ & $740$ & $772 $ \\ 
\hline
Distorted & $498$ & $502$ & $797$ & $818$  \\
\hline
Cylindrical & $1626$ & $2977$ & $2583$& $4807$\\
\hline 
Elliptical & $2270$ & $2881$ & $3805$& $ 4678$ \\
\hline
\hline
\multirow{3}{*} {Grid} & \multicolumn{4}{|c|}{ Number of iterations PCG solver} \\
\cline{2-5}
&  \multicolumn{2}{|c|}{ $p_x=2$} & \multicolumn{2}{|c|}{ $p_x=3$ }\\
\cline{2-5}
& $ N_x=8 $  &  $N_x =32$ & $ N_x=8 $ & $N_x=32$ \\
\hline
Cartesian &  $8$ & $8$ & $11$ & $11$\\ 
\hline
Distorted  & $18$& $18$ & $23$ & $22$\\
\hline
Cylindrical & $10$ & $10$& $13$ & $14$\\
\hline 
Elliptical & $14$ &$16$ & $21$ &$22$\\
\hline
\end{tabular}
\label{table:pcgiterations}
\end{table}

Table \ref{table:pcgiterations} shows the maximum number of iterations for the CG and PCG solver of the mass matrices. The numbers are taken from the simulation of the electromagnetic Weibel instability in Section \ref{sec:numericsboundary} on various grids with a time step of $\Delta t = 0.01$ and a solver tolerance of $10^{-13}$. Note that for a lower tolerance the CG solver would not converge without the preconditioner. It can be seen that our preconditioner largely reduces the number of iterations and the iteration count only moderately increases with increasing spline order.

\section{Numerical Experiments} \label{sec:numericsboundary}
In this section, we demonstrate our algorithms on a number of transformed domains. The computations are performed with our implementation of the GEMPIC framework within the Fortran library SeLaLib \cite{selalib}.
\subsection{Coordinate Transformations}
In this subsection, we introduce various coordinate transformations that will be used in our numerical experiments. 
First, we present a sinusoidal transformation of the form
\begin{align}\label{dist_deform}
F_{dist}(\xib) = \begin{pmatrix}
L_x\left(\xi_1 + \epsilon \sin(L_p \xi_1) \sin(2\pi \xi_2)\right)\\
L_y\left(\xi_2 + \epsilon \sin(L_p \xi_1) \sin(2\pi \xi_2)\right)\\
L_z \xi_3
\end{pmatrix}.
\end{align}
\begin{figure}[!ht]
\centering
 \begin{subfigure}{0.4\textwidth}
  \centering
  \includegraphics[width=.6\linewidth]{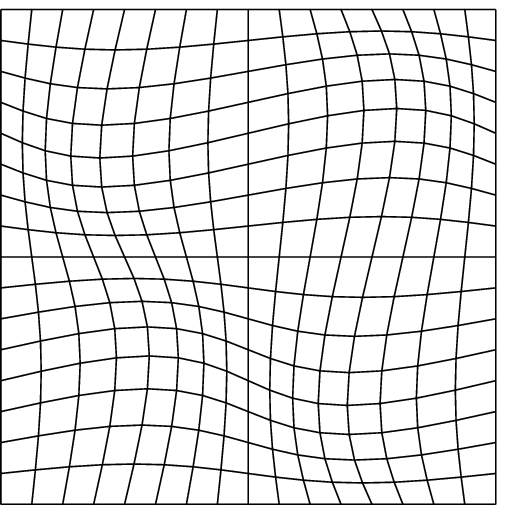}
  \caption{Square domain}
 \end{subfigure}
 \hspace{0.5cm}
 \begin{subfigure}{0.4\textwidth}
  \centering
  \includegraphics[width=.6\linewidth]{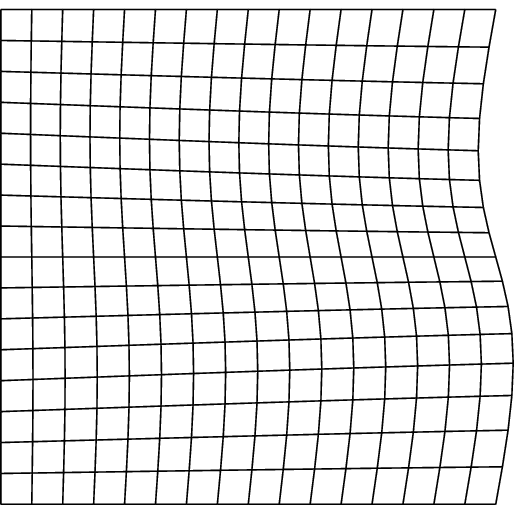}
  \caption{Deformed domain}  
 \end{subfigure}
 \caption{ Distorted grids on different domains for distortion parameter $\epsilon = 0.05$.}
  \label{fig:colboundmesh}	
\end{figure}  
 Figure \ref{fig:colboundmesh}  visualises the $(x,y)$-part of the sinusoidally distorted grid with a distortion parameter of $\epsilon=0.05$ on a square domain with $L_p=2\pi$ and on a deformed domain with $L_p= \frac{\pi}{2}$. 

Furthermore, we introduce two radial mappings, a cylindrical and an elliptical transformation,
{\small
\begin{align}\label{cylindrical}
F_{cyl}(\xib) = \begin{pmatrix}
(r_0 + L_r \xi_1) \cos(2\pi \xi_2)\\
(r_0 + L_r \xi_1) \sin(2\pi \xi_2)\\
L_z \xi_3
\end{pmatrix}, \quad
F_{ell}(\xib) =  
\begin{pmatrix}
L_r\cosh(\xi_1+r_0) \cos(2\pi \xi_2)\\
L_r\sinh(\xi_1+r_0) \sin(2\pi \xi_2)\\
L_z \xi_3
\end{pmatrix}.
\end{align}
}
\begin{figure}[!ht]
\centering
 \begin{subfigure}{0.45\textwidth}
  \centering
  \includegraphics[width=\linewidth]{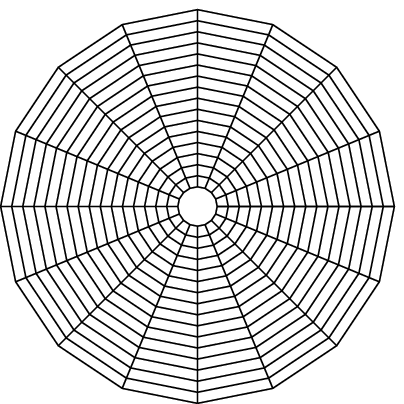}
  \caption{Cylindrical grid}
 \end{subfigure}
 \hspace{0.5cm}
 \begin{subfigure}{0.45\textwidth}
  \centering
  \includegraphics[width=\linewidth]{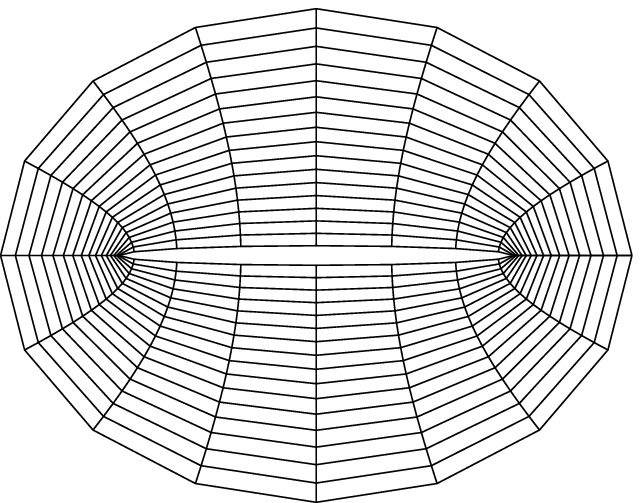}
  \caption{Elliptical grid}  
 \end{subfigure}
 \caption{ Cylindrical grid with $r_0=0.5$ and elliptical grid with $r_0=0.05$.}
 \label{fig:radialmeshes}
\end{figure}  

Figure \ref{fig:radialmeshes} visualises the $(x,y)$-part of the corresponding grids, where the pole is excluded for $r_0>0$ avoiding a singularity in the mapping.

\subsection{Test case}
Motivated by the results in \cite{chacon2019}, we test the implementation of the perfect conductor boundary and the reflecting particle boundary conditions with a simulation of the Weibel instability \cite{weibel1959spontaneously}. The instability is excited by an anisotropy in the thermal velocity and amplified with the initialisation of the corresponding component of the magnetic field. The initial distribution is given by
{\small
\begin{align*}
f(\xb,\vb,t=0)= \left(1+\alpha \cos( \kb \cdot \xb)\right)\frac{1}{(2\pi)^\frac{3}{2} v_{Tx}v_{Ty}v_{Tz}} \exp \left( -\frac{1}{2} \left( \frac{v_x^2}{v_{Tx}^2}+\frac{v_y^2}{v_{Ty}^2}+\frac{v_z^2}{v_{Tz}^2}\right) \right ),
\end{align*}
}
where  $\xb \in \left[0,\frac{2 \pi}{1.25}\right]^3, \vb \in \R^3$.
We can choose between the following three scenarios that trigger the instability:
\begin{itemize}
\item $\kb = 1.25 \hat{\textbf{e}}_x, v_{Tx}< v_{Ty,z}$, where we initialise 
\begin{align*}
B_2(x) = \beta \cos(k_x x)  \text{ or }  B_3(x) = \beta \cos(k_x x),
\end{align*}

\item $\kb = 1.25 \hat{\textbf{e}}_y, v_{Ty}< v_{Tx,z}$, where we initialise  
\begin{align*}
B_1(x,y) = \beta \cos(k_y y) \sin\left(\frac{\pi x}{L_x}\right) 
 \text{ or }  B_3(y) = \beta \cos(k_y y), 
\end{align*}

\item $\kb = 1.25 \hat{\textbf{e}}_z, v_{Tz}< v_{Tx,y}$, where we initialise  
\begin{align*}
B_1(x,z) = \beta \cos(k_z z) \sin\left(\frac{\pi x}{L_x}\right) 
 \text{ or } B_2(z) = \beta \cos(k_z z). 
\end{align*}
\end{itemize}
The remaining two components of the magnetic field are initialised with zero.

We set $v_{Ti}=\frac{0.02}{\sqrt{2}} = \frac{v_{Tj}}{\sqrt{12}}= \frac{v_{Tk}}{\sqrt{12}},$ where $(i,j,k) = (x,y,z), (y,z,x) $ or $(z,x,y)$. To start above the particle noise, we set $\beta=10^{-3}$ for the initial magnetic field. The initial electric field is calculated from Poisson's equation and the initial perturbation in space is set to zero with $\alpha=0$.

For the numerical resolution, we take $2{,}048{,}000$ particles, $8{\times}8{\times}8$ grid cells, cubic splines and a time step of $\Delta t = 0.1 $. The tolerance of the iterative solvers for the DisGradE method is set to $10^{-13}$ and the tolerance of the PCG solver for the mass matrices is set to $10^{-14}$. Note that we normalised to dimensionless quantities in terms of the electron Debye length $\lambda_{De}$ and the plasma frequency $\omega_{pe}$.

\subsection{Comparison of Perfect Boundary with Periodic Boundary Conditions}
We start by comparing simulation results with the perfect conductor boundary conditions to the simulation results with periodic boundary conditions.
\begin{figure}[!ht]
\centering
 \begin{subfigure}{0.48\textwidth}
  \centering
  \includegraphics[width=\linewidth]{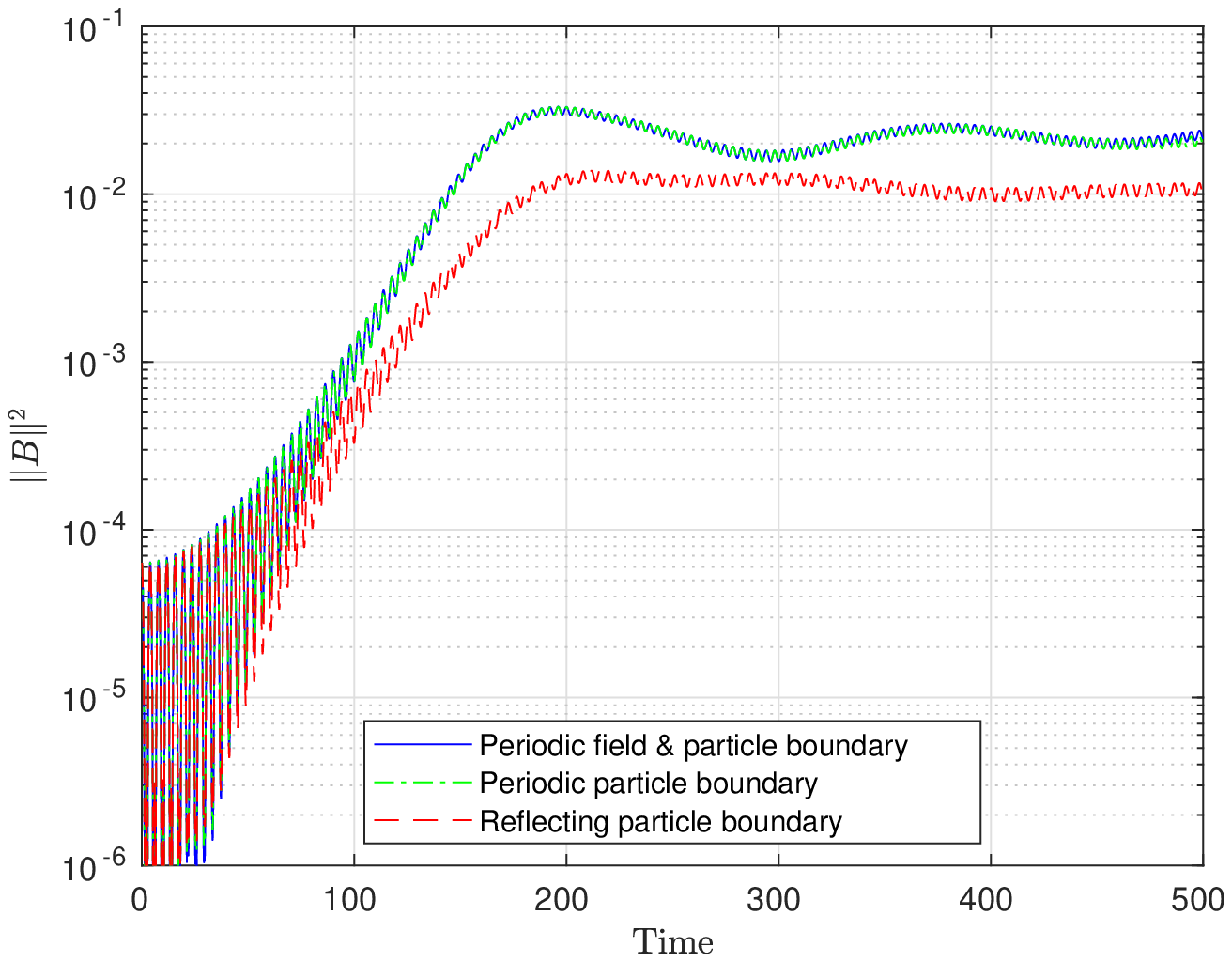}
\caption{Initialisation of $B_2(x)= \beta \cos(k_x x)$}
\label{fig:Bcartesian_x_By} 
 \end{subfigure}
 \hfill
 \begin{subfigure}{0.48\textwidth}
  \centering
  \includegraphics[width=\linewidth]{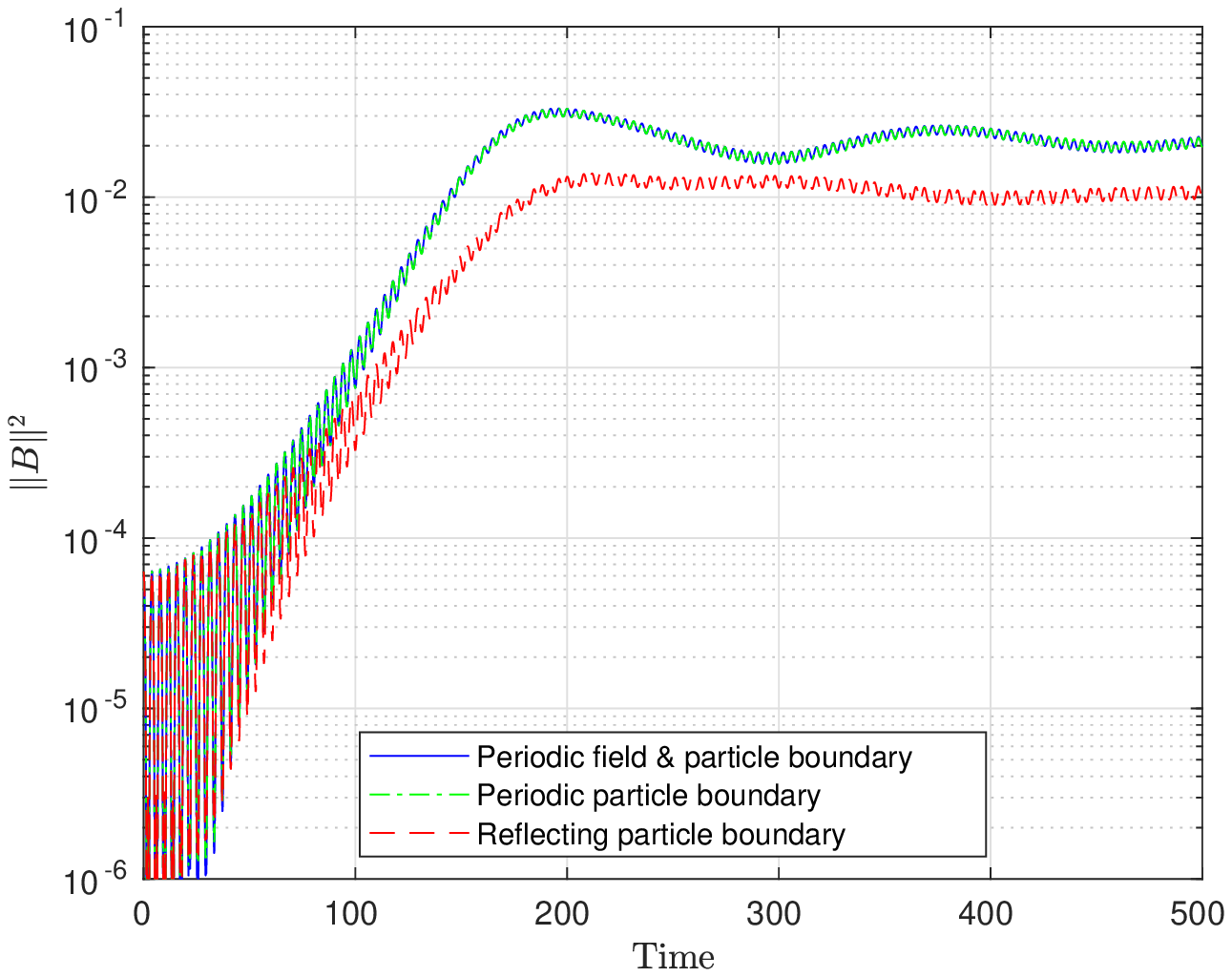}
\caption{Initialisation of $B_3(x)= \beta \cos(k_x x)$}
\label{fig:Bcartesian_x_Bz} 
 \end{subfigure}
 \caption{Weibel instability with $\kb = 1.25 \hat{\textbf{e}}_x$: Magnetic field energy for HS with time step $\Delta t = 0.1$ on a Cartesian grid with different boundary conditions.} 
 \label{fig:Bcartesian_x}
\end{figure}

\begin{figure}[!ht]
\centering
 \begin{subfigure}{0.48\textwidth}
  \centering
  \includegraphics[width=\linewidth]{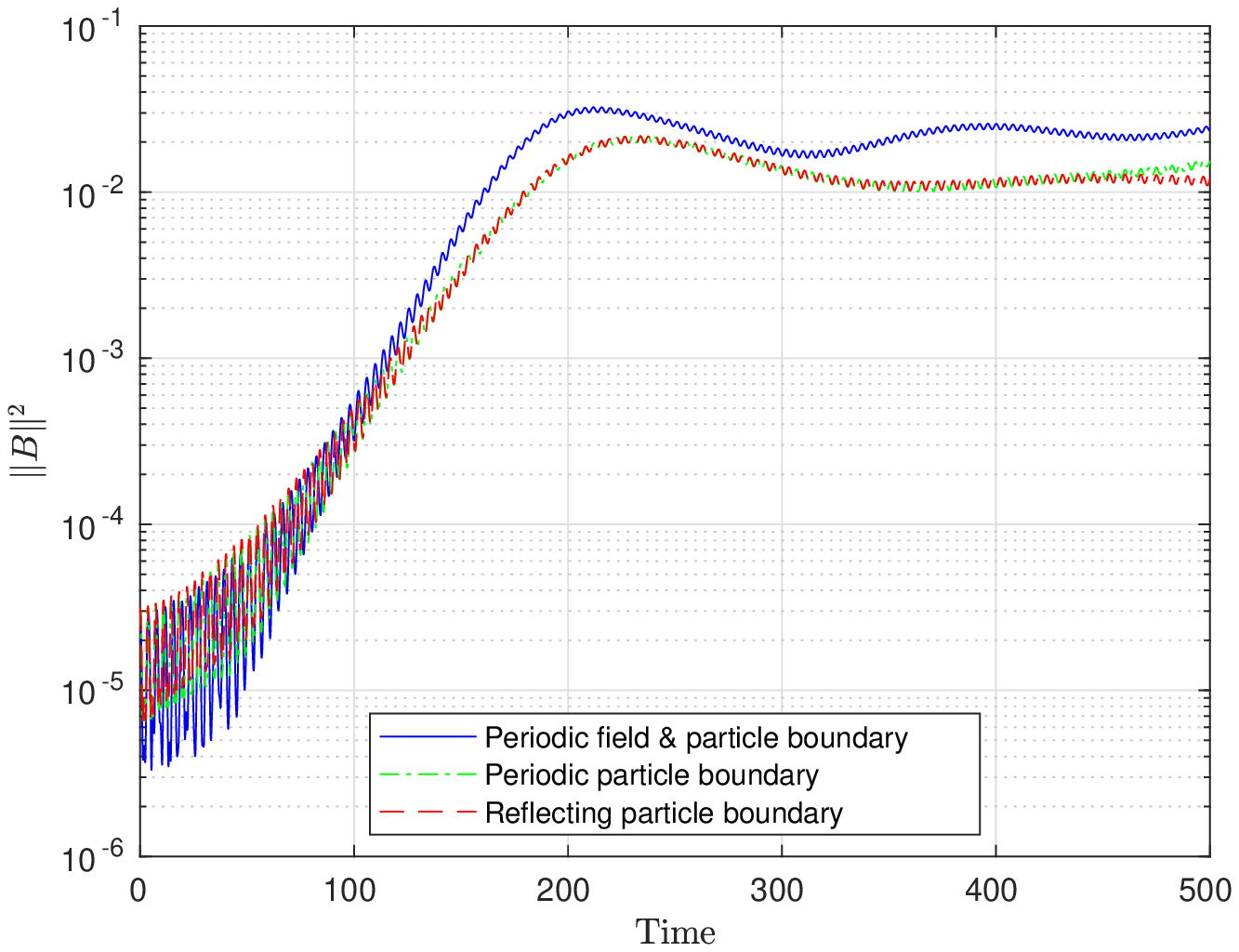}
\caption{Initialisation of $B_1=\beta \sin\left(\frac{\pi x}{L_x}\right) \cos(k_y y) $}
\label{fig:Bcartesian_y_Bx} 
 \end{subfigure}
 \hfill
 \begin{subfigure}{0.48\textwidth}
  \centering
  \includegraphics[width=\linewidth]{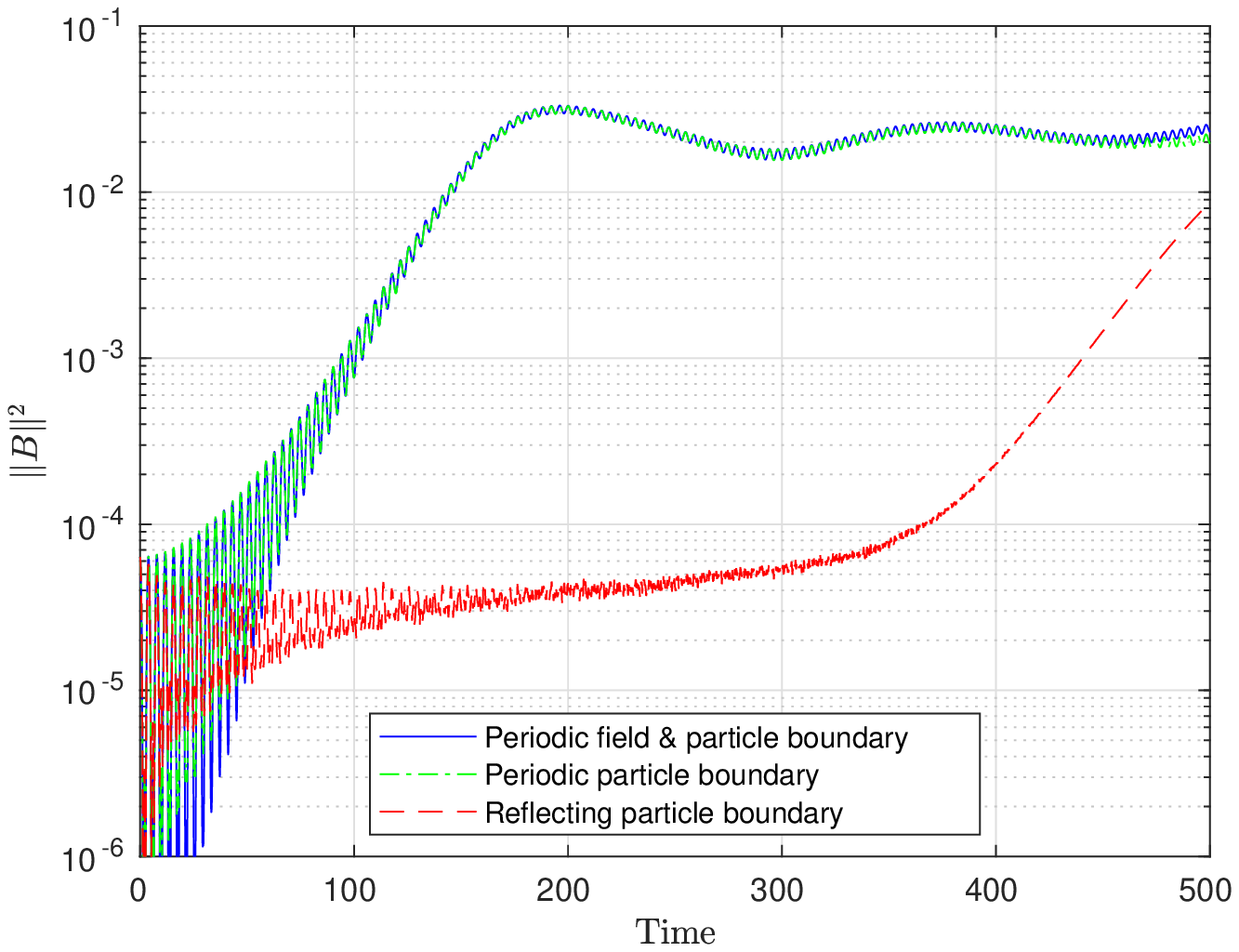}
\caption{Initialisation of $B_3(y)= \beta \cos(k_y y)$}
\label{fig:Bcartesian_y_Bz} 
 \end{subfigure}
 \caption{Weibel instability with $\kb = 1.25 \hat{\textbf{e}}_y$: Magnetic field energy for HS with time step $\Delta t = 0.1$ on a Cartesian grid with different boundary conditions.} 
 \label{fig:Bcartesian_y}
\end{figure}

\begin{figure}[!ht]
\centering
 \begin{subfigure}{0.48\textwidth}
  \centering
  \includegraphics[width=\linewidth]{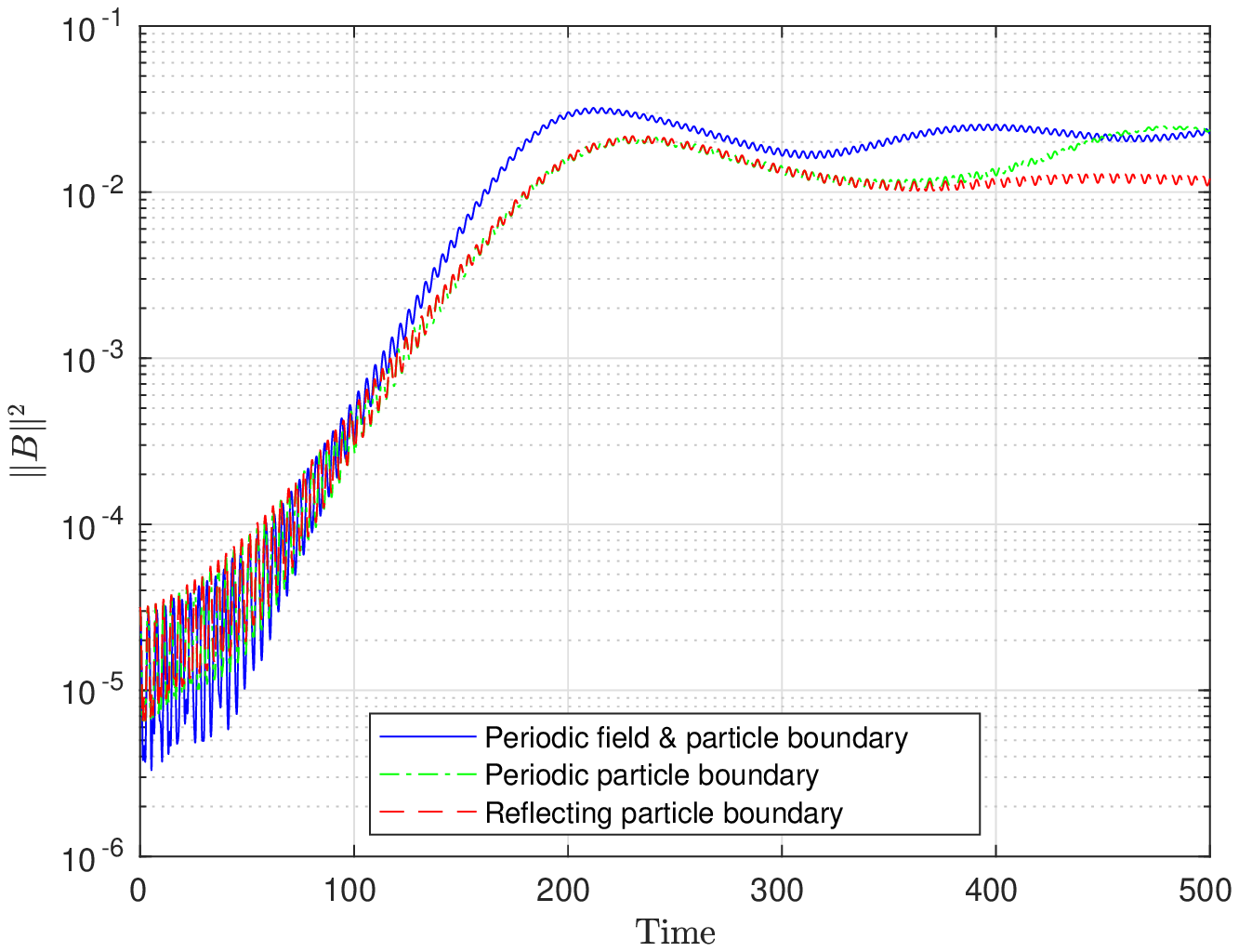}
\caption{Initialisation of $B_1=\beta \sin\left(\frac{\pi x}{L_x}\right) \cos(k_z z)$}
\label{fig:Bcartesian_z_Bx} 
 \end{subfigure}
 \hfill
 \begin{subfigure}{0.48\textwidth}
  \centering
  \includegraphics[width=\linewidth]{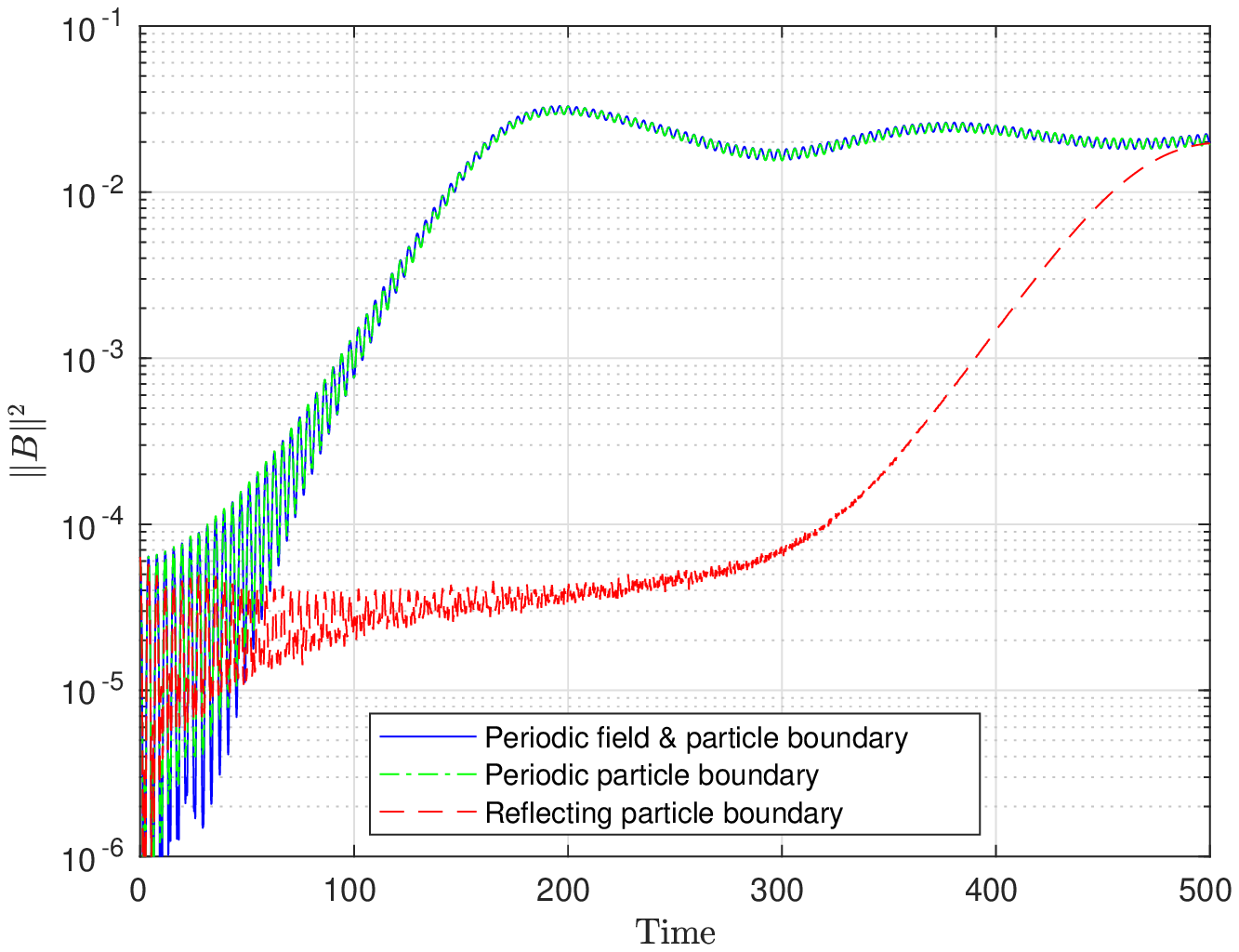}
 \caption{Initialisation of $B_2(z)= \beta \cos(k_z z)$}
\label{fig:Bcartesian_z_By} 
 \end{subfigure}
 \caption{Weibel instability with $\kb = 1.25 \hat{\textbf{e}}_z$: Magnetic field energy for HS with time step $\Delta t = 0.1$ on a Cartesian grid with different boundary conditions.} 
  \label{fig:Bcartesian_z}
\end{figure}

\begin{figure}[!ht]
\centering
 \begin{subfigure}{0.48\textwidth}
  \centering
  \includegraphics[width=\linewidth]{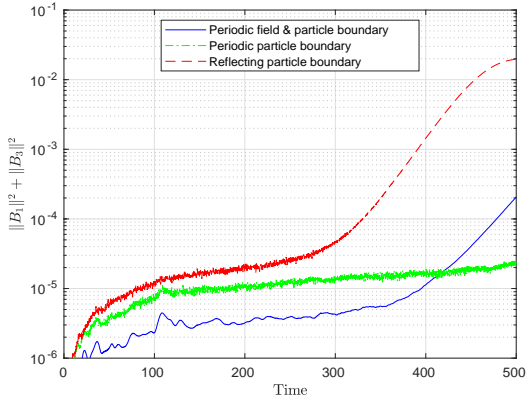}
\caption{ Sum of first and third component of the magnetic field }
\label{fig:Bcartesian_z_By_comp13} 
 \end{subfigure}
 \hfill
 \begin{subfigure}{0.48\textwidth}
  \centering
  \includegraphics[width=\linewidth]{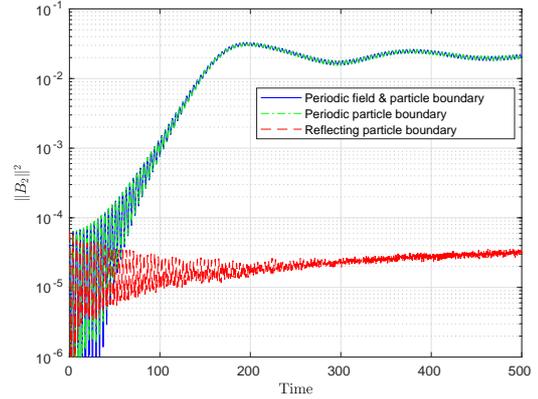}
 \caption{Second component of the magnetic field}
\label{fig:Bcartesian_z_By_comp2} 
 \end{subfigure}
 \caption{Weibel instability with $\kb = 1.25 \hat{\textbf{e}}_z$ and initialisation of $B_2$: Different components of the magnetic field energy for HS with time step $\Delta t = 0.1$ on a Cartesian grid with different boundary conditions.} 
\end{figure}

Figures \ref{fig:Bcartesian_x}, \ref{fig:Bcartesian_y} and \ref{fig:Bcartesian_z} show the magnetic field energy on a Cartesian grid for the three different choices of the wave vector $\kb=1.25 \hat{\textbf{e}}_{i}, i\in\{x,y,z\}$. Additionally, there are two different components of the magnetic field in each scenario that can be initialised to start the Weibel instability right away. Since the results coincide for the different time integrators, we show only the simulation of the semi-explicit HS scheme with a time step of $\Delta t=0.1$. Displayed are the simulation results for the perfect conductor boundary conditions with periodic or reflecting particle boundaries and a simulation with periodic field and particle boundary conditions for comparison. 

When initialising the second and third component of the magnetic field, we see in Figures \ref{fig:Bcartesian_x}, \ref{fig:Bcartesian_y_Bz} and \ref{fig:Bcartesian_z_By} that the growth of the magnetic field in the simulation with the periodic particle boundary coincides with the growth of the magnetic field in the simulation with periodic field and particle boundary conditions. However, the reflecting particle boundary leads to a lower growth rate and delays the beginning of the growth in the magnetic field in the two cases displayed in Figures \ref{fig:Bcartesian_y_Bz} and \ref{fig:Bcartesian_z_By}. We will try to explain this behaviour exemplarily for the latter case. 
In Figure \ref{fig:Bcartesian_z_By_comp2}, we see that the second component of the magnetic field, which was initialised to start the instability right away, shows no signs of the expected growth for the reflecting boundary conditions. This component is updated as $\partial_t {B_2}=\partial_z E_1 - \partial_x E_3$ in Faraday's law \eqref{Faraday}. Normally, the anisotropy in the velocity causes a discrepancy between the partial derivatives of these two components of the electric field resulting in the growth of the magnetic field. However, the reflecting boundary conditions prevent a current through the boundary, which seems to level the values of the two components of the electric field.
Therefore, the growth of the  second component of the magnetic field is suppressed. Nevertheless, Figure \ref{fig:Bcartesian_z_By_comp13} shows that the instability arises in the first and third component of the magnetic field. 
However, the growth is delayed, since these two components were not initialised.
 
On a Cartesian domain, the physical and the logical fields only differ by a constant scaling. Therefore, the perfect conductor boundary conditions on the logical fields ensure $E_2,E_3$ and $B_1$ to be zero at the boundary. This is why, we initialise the first component of the magnetic field with $B_1(\xb) = \beta \sin\left(\frac{\pi x}{L}\right) \cos(k_y y + k_z z)$. In this case, the growth rates of the periodic and reflecting particle boundary conditions coincide, which are lower than the one with periodic field and particle boundary conditions as can be seen in Figures \ref{fig:Bcartesian_y_Bx} and \ref{fig:Bcartesian_z_Bx}.

\subsection{Domain Deformation}
Next, we are interested in the behaviour on a deformed mapped grid. Therefore, we apply the sinusoidal coordinate transformation \eqref{dist_deform} with $L_p=\frac{\pi}{2}$. Again the results for the different time integrators coincide so that we only show the results of the semi-explicit HS scheme with a time step of $\Delta t = 0.1$. 

\begin{figure}[!ht]
\centering
 \begin{subfigure}{0.48\textwidth}
  \centering
  \includegraphics[width=\linewidth]{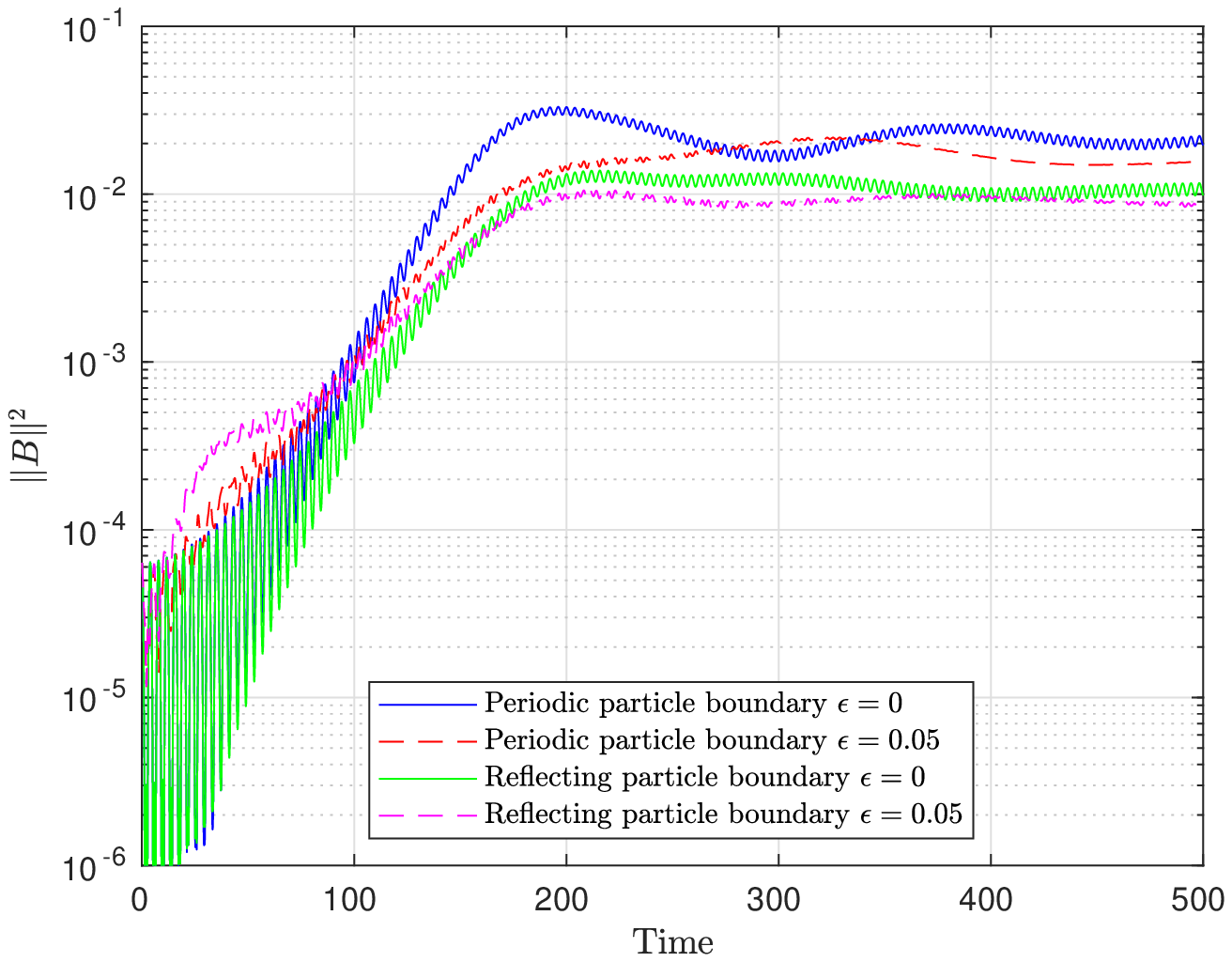}
\caption{Initialisation of $B_2(x)= \beta \cos(k_x x)$}
\label{fig:Btrafo_x_By} 
 \end{subfigure}
 \hfill
 \begin{subfigure}{0.48\textwidth}
  \centering
  \includegraphics[width=\linewidth]{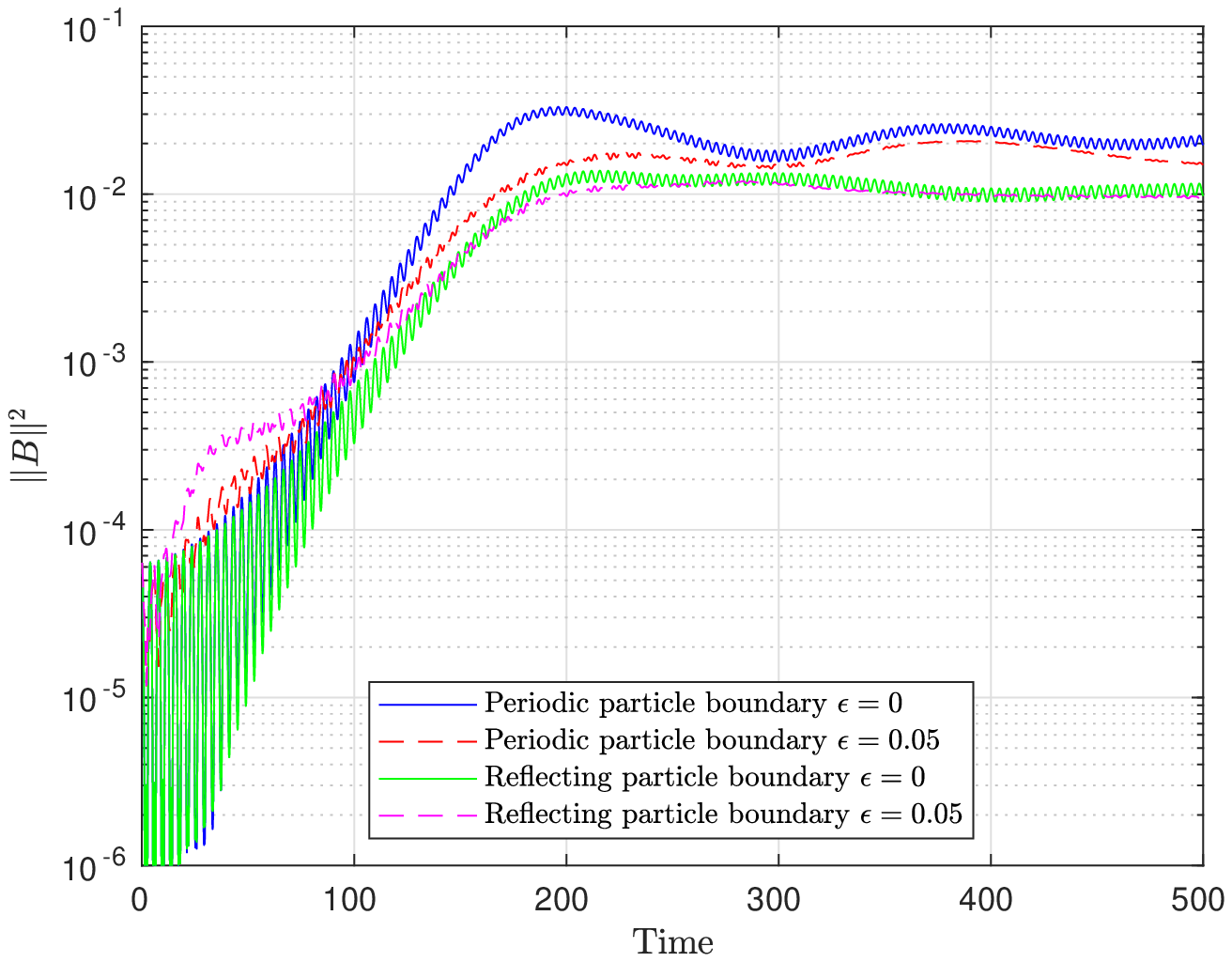}
\caption{Initialisation of $B_3(x)= \beta \cos(k_x x)$}
\label{fig:Btrafo_x_Bz} 
 \end{subfigure}
 \caption{Weibel instability with $\kb = 1.25 \hat{\textbf{e}}_x$: Magnetic field energy for HS with time step $\Delta t = 0.1$ on a distorted grid with distortion parameters $\epsilon=0,0.05$ for the coordinate transformation.} 
 \label{fig:Btrafo_x}
\end{figure}

\begin{figure}[!ht]
\centering
 \begin{subfigure}{0.48\textwidth}
  \centering
  \includegraphics[width=\linewidth]{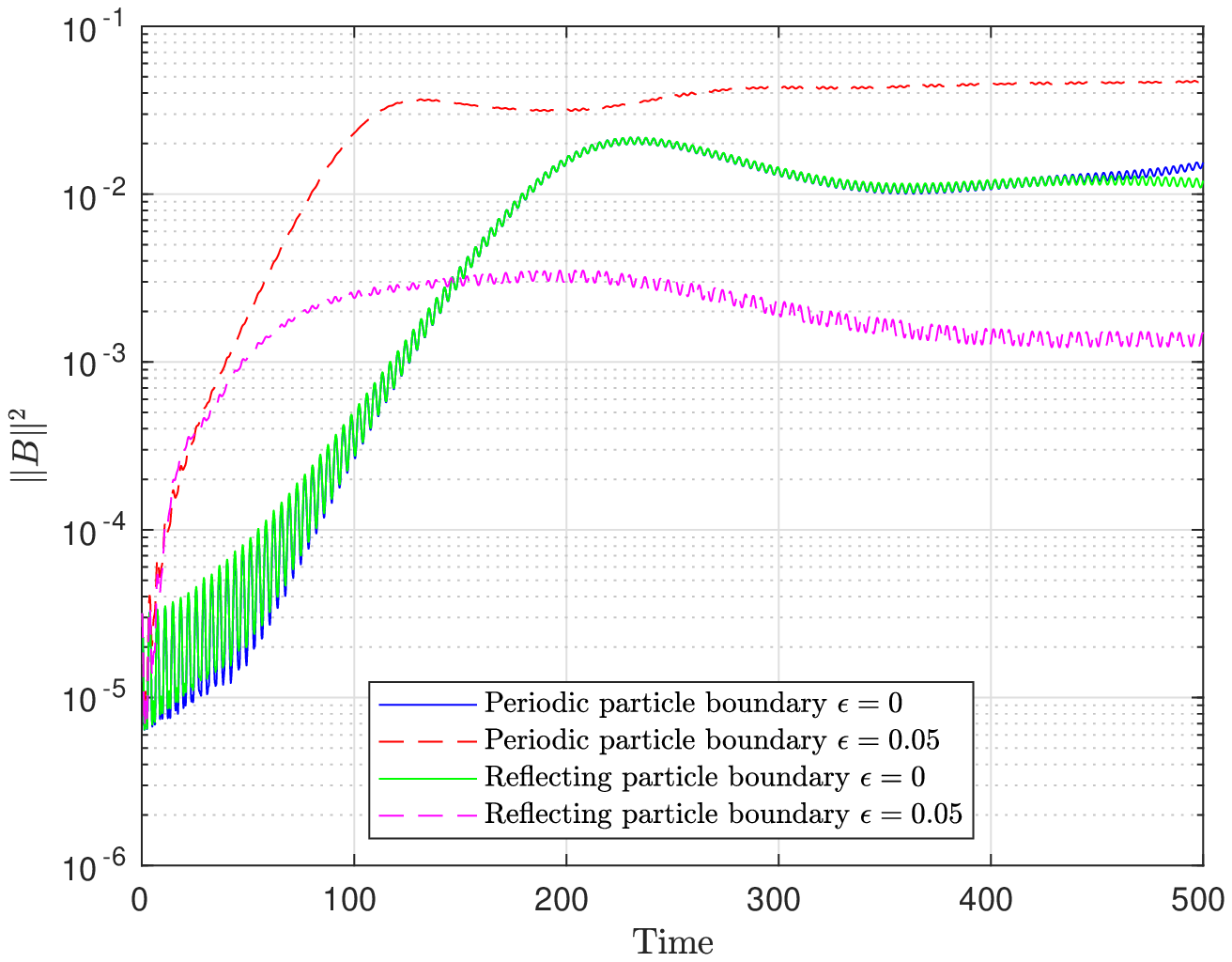}
\caption{Initialisation of $B_1=\beta \sin\left(\frac{\pi x}{L_x}\right) \cos(k_y y) $}
\label{fig:Btrafo_y_Bx} 
 \end{subfigure}
 \hfill
 \begin{subfigure}{0.48\textwidth}
  \centering
  \includegraphics[width=\linewidth]{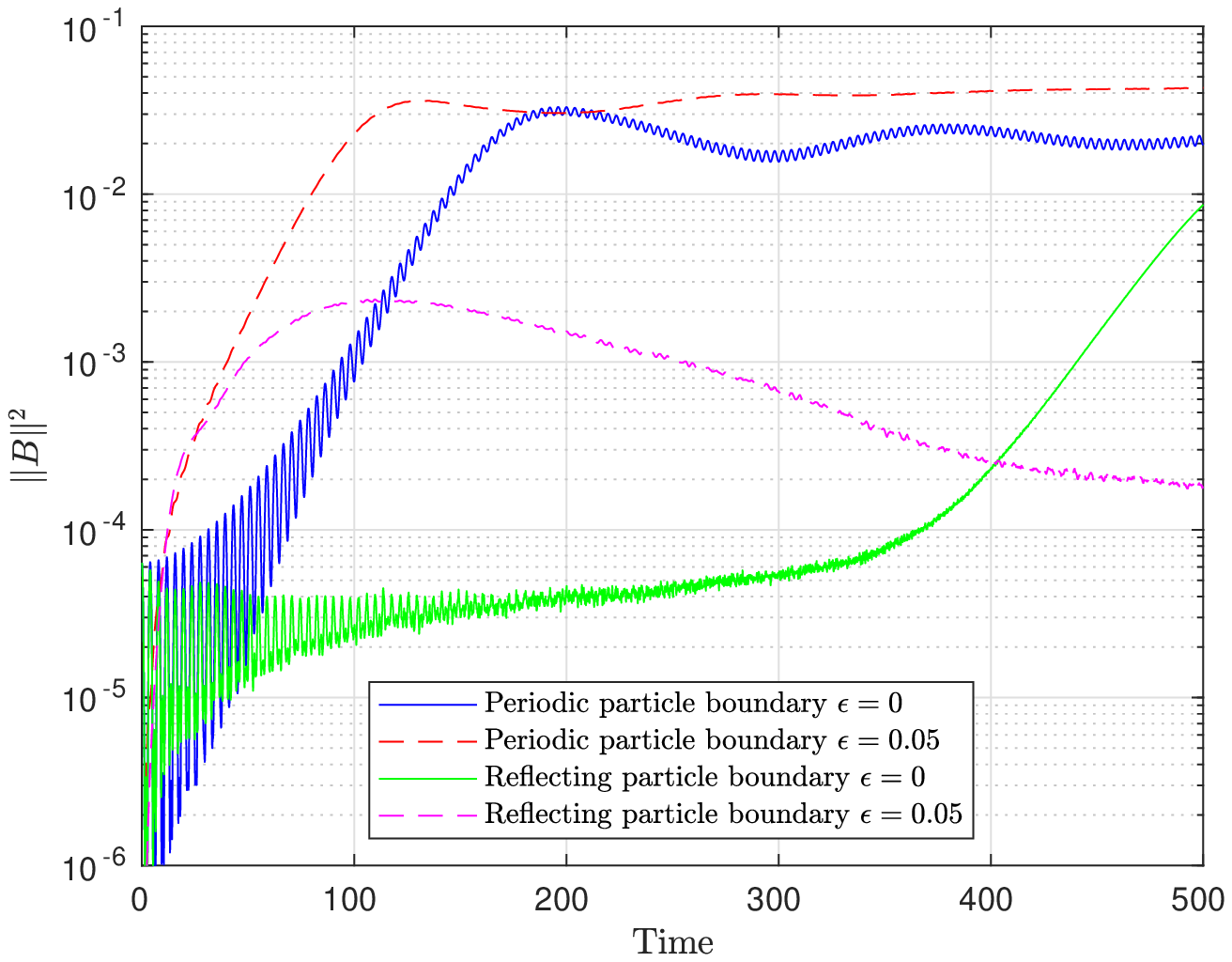}
\caption{Initialisation of $B_3(y)= \beta \cos(k_y y)$}
\label{fig:Btrafo_y_Bz} 
 \end{subfigure}
 \caption{Weibel instability with $\kb = 1.25 \hat{\textbf{e}}_y$: Magnetic field energy for HS with time step $\Delta t = 0.1$ on a distorted grid with distortion parameters $\epsilon=0,0.05$ for the coordinate transformation.} 
 \label{fig:Btrafo_y}
\end{figure}

\begin{figure}[!ht]
\centering
 \begin{subfigure}{0.48\textwidth}
  \centering
  \includegraphics[width=\linewidth]{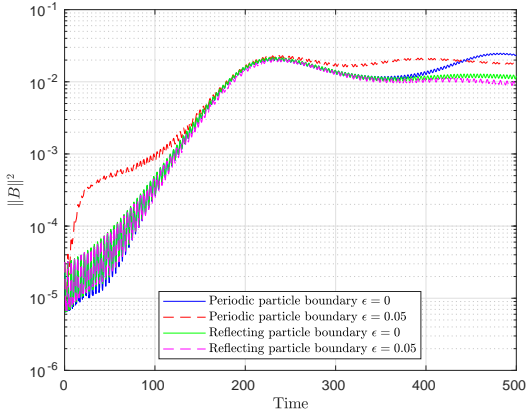}
\caption{Initialisation of $B_1=\beta \sin\left(\frac{\pi x}{L_x}\right) \cos(k_z z)$}
\label{fig:Btrafo_z_Bx} 
 \end{subfigure}
 \hfill
 \begin{subfigure}{0.48\textwidth}
  \centering
  \includegraphics[width=\linewidth]{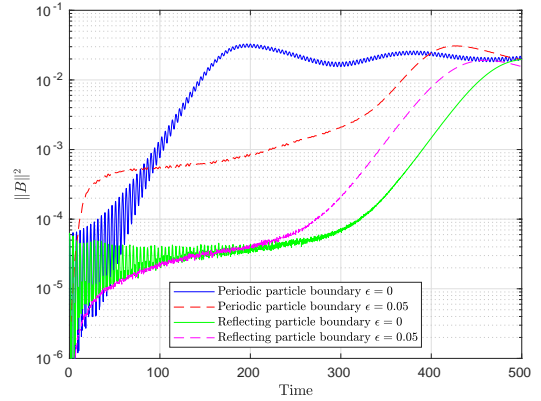}
 \caption{Initialisation of $B_2(z)= \beta \cos(k_z z)$}
\label{fig:Btrafo_z_By} 
 \end{subfigure}
 \caption{Weibel instability with $\kb = 1.25 \hat{\textbf{e}}_z$: Magnetic field energy for HS with time step $\Delta t = 0.1$ on a distorted grid with distortion parameters $\epsilon=0,0.05$ for the coordinate transformation.} 
  \label{fig:Btrafo_z}
\end{figure}
Figures \ref{fig:Btrafo_x}, \ref{fig:Btrafo_y} and \ref{fig:Btrafo_z} show the magnetic field energy on the distorted grid with distortion parameters $\epsilon=0$, which is equal to the Cartesian grid, and $\epsilon=0.05$ for the three different choices of the wave vector $\kb=1.25 \hat{\textbf{e}}_{i}, i\in\{x,y,z\}$. Displayed are the simulation results for the perfect conductor boundary conditions with periodic and with reflecting particle boundaries. Figure \ref{fig:Btrafo_x} shows that the domain deformation effects the growth of the magnetic field especially in the beginning. In Figure \ref{fig:Btrafo_y}, we see that the coordinate transformation couples the coordinate directions, which leads to the initial steep growth in the magnetic field. In Figure \ref{fig:Btrafo_z_By}, the reflecting particle boundary still delays the growth of the magnetic field but it starts earlier on the deformed grid. For the periodic particle boundary, the domain deforming transformation leads to a steeper growth in the beginning, which can be seen in Figures \ref{fig:Btrafo_z_Bx} and  \ref{fig:Btrafo_z_By}. Figure \ref{fig:Btrafo_z_Bx} shows only minor differences for the simulation with a reflecting particle boundary on the Cartesian or on the deformed grid.

Since the coordinate transformation mixes the first two directions of the logical coordinates in the computation of $x$ and $y$, $z$ is the only periodic coordinate direction left. Therefore, we focus on the scenario with the wave vector $\kb=1.25\hat{\textbf{e}}_z$.
Since the instability is delayed under coordinate transformation when we initialise the second component of the magnetic field, we look at the case, where the first component of the magnetic field is initialised with $B_1(x,z)=\beta \sin\left(\frac{\pi x}{L_x}\right) \cos(k_z z)$.  

As a next step, we investigate the effect of the distortion parameter $\epsilon$ and compare the distorted grid \eqref{dist_deform} on a square mapped grid with $L_p=2\pi$ to the deformed mapped grid with $L_p=\frac{\pi}{2}$.
\begin{figure}[!ht]
\centering
 \begin{subfigure}{0.48\textwidth}
  \centering
  \includegraphics[width=\linewidth]{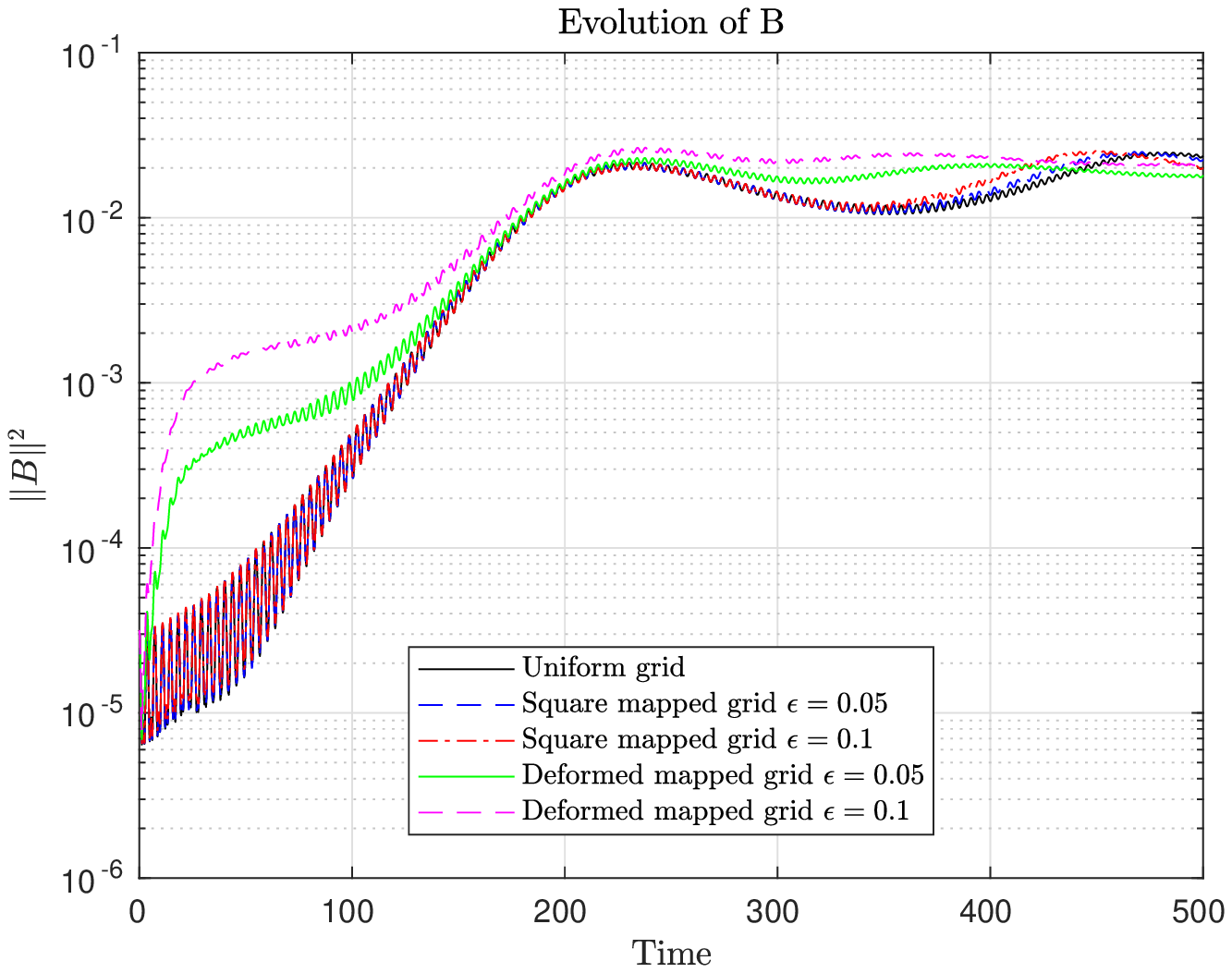}
\caption{Periodic particle boundary}
 \end{subfigure}
 \hfill
 \begin{subfigure}{0.48\textwidth}
  \centering
  \includegraphics[width=\linewidth]{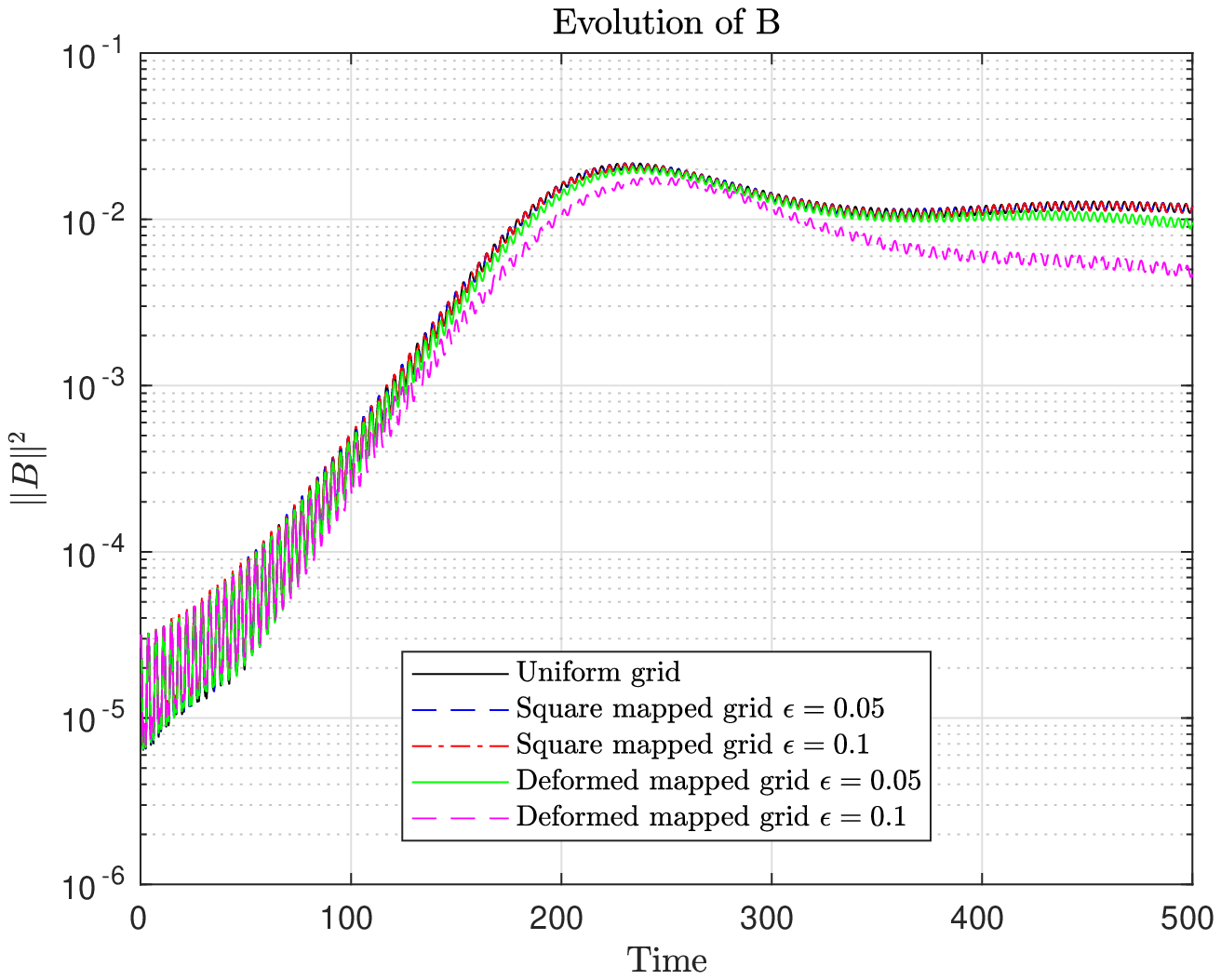}
\caption{Reflecting particle boundary}
 \end{subfigure}
 \caption{Weibel instability with $\kb = 1.25 \hat{\textbf{e}}_z$: Magnetic field energy for HS with time step $\Delta t = 0.1$ on a distorted grid with different values of the distortion parameter $\epsilon$ for the coordinate transformation.} 
 \label{fig:weibel_dist}
\end{figure}

In Figure \ref{fig:weibel_dist}, the magnetic field energy on the distorted grid is displayed for different values of the distortion parameter $\epsilon$. We present only the simulation results of the HS scheme with a time step of $\Delta t=0.1$ because all the schemes show the same behaviour of the magnetic field. We see that for the square mapped grid, the coordinate transformation does not change the growth of the magnetic field even for a high distortion parameter of $\epsilon=0.1$. However, for the domain deforming coordinate transformation the higher distortion parameter $\epsilon=0.1$ leads to a slightly different growth of the magnetic field. Additionally, after the saturation all the curves drift apart, especially with the periodic particle boundary conditions.

\subsection{Radial Grids}
Let us extend the Weibel instability to the radial grids given by the coordinate transformations in \eqref{cylindrical}, where we set $r_0=0.01$ to prevent a singularity at $\xi_1=0$ and choose $L_r=L-r_0$. From now on, we consider the particles to be reflected at the boundary, since this is an expected physical behaviour. Additionally, we double the resolution in the radial and angular directions to $16{\times}16{\times}8$ grid cells in order to see the impact of the time step constraints on the semi-explicit integrators.

\begin{figure}[!ht]
\centering
 \begin{subfigure}{0.48\textwidth}
  \centering
  \includegraphics[width=\linewidth]{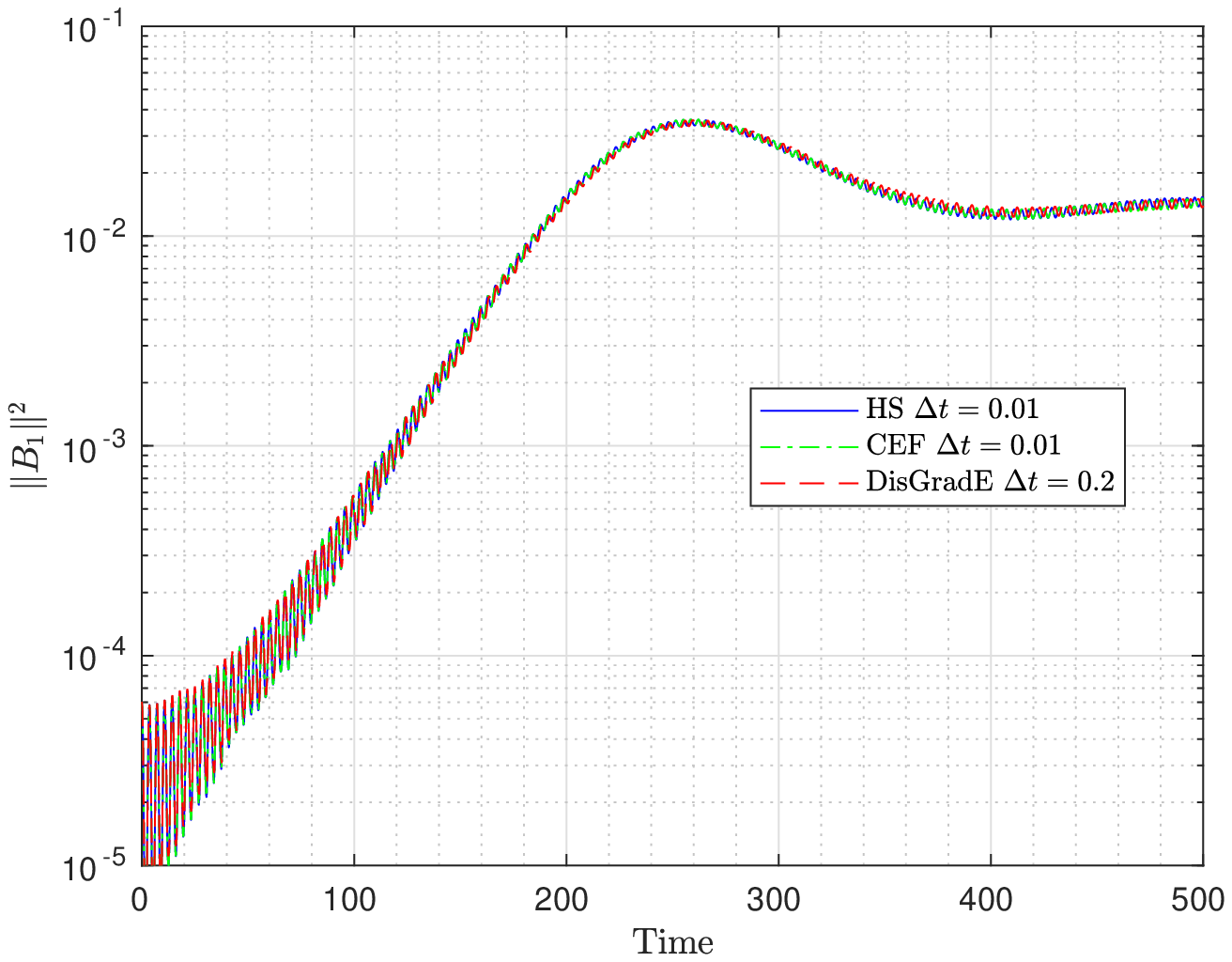}
\caption{Cylindrical grid}
 \end{subfigure}
 \hfill
 \begin{subfigure}{0.48\textwidth}
  \centering
  \includegraphics[width=\linewidth]{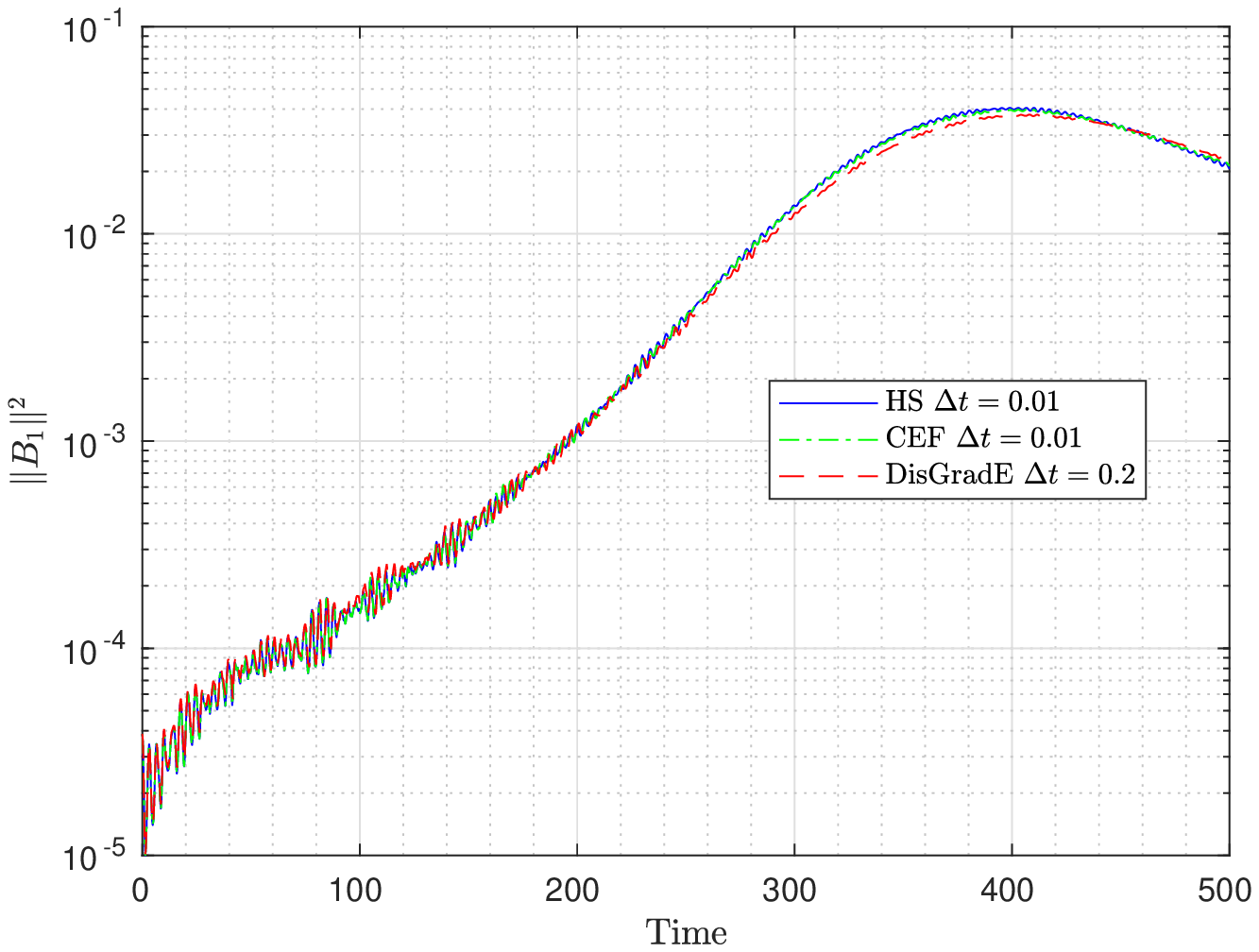}
\caption{Elliptical grid}
 \end{subfigure}
 \caption{Weibel instability with $\kb = 1.25 \hat{\textbf{e}}_z$: First component of the magnetic field energy for various integrators with different time steps on cylindrical and elliptical grids with $r_0=0.01$.} 
 \label{fig:weibel_radial}
\end{figure}
Figure \ref{fig:weibel_radial} shows the first component of the magnetic field energy on a cylindrical and an elliptical grid with reflecting particle boundary conditions. Due to the stability constraints, we have to choose a lower time step of $\Delta t = 0.01$ for the semi-explicit schemes whereas the semi-implicit method shows comparable results with a time step of $\Delta t = 0.2$ or higher. The stability constraints arise from the smaller cells near the pole, where the semi-explicit schemes have problems, when particles cross too many cells in one time step.
Although the semi-implicit DisGradE method takes roughly about ten times longer for a time step than the semi-explicit HS or CEF schemes, it seems that it is more suitable to this type of domain deforming mappings, since its simulation results show the same behaviour of the magnetic field as the semi-explicit schemes while requiring at most one-twentieth of the time steps.

\begin{figure}[!ht]
\centering
\includegraphics[width=0.85\linewidth]{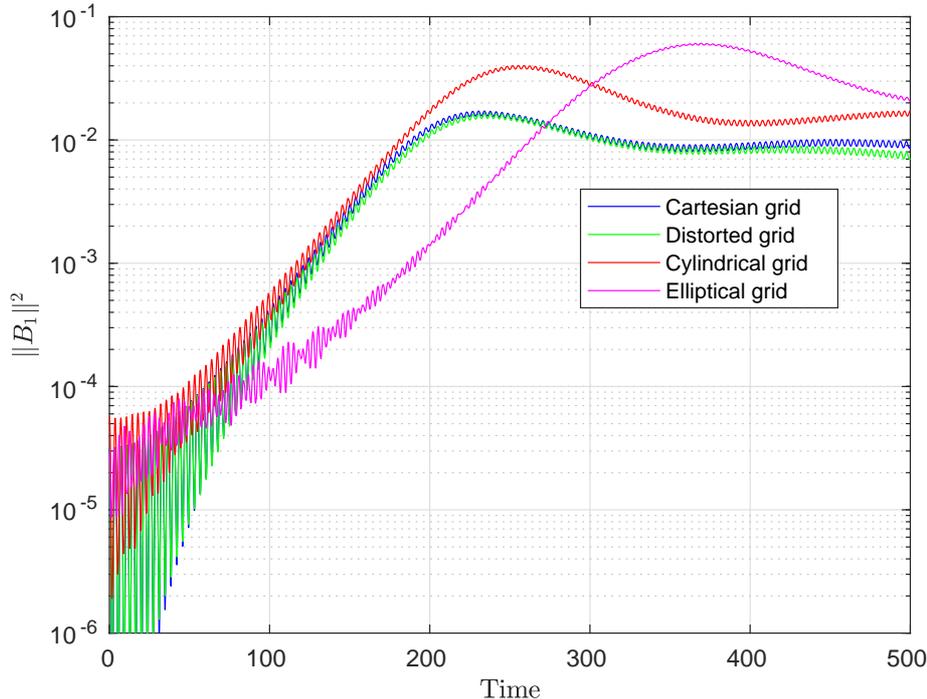}
\caption{Weibel instability with $\kb = 1.25 \hat{\textbf{e}}_z$: First component of the magnetic field energy for DisGradE with time step $\Delta t = 0.1$ for various domain deforming mappings.} 
 \label{fig:weibel_domaindeforming}
\end{figure}
In Figure \ref{fig:weibel_domaindeforming}, we see the first component of the magnetic field initialised with $B_1(x,z)=\beta \sin\left(\frac{\pi x}{L_x}\right) \cos(k_z z)$ for various domain deforming mappings. The simulation results are obtained by the DisGradE method with a time step of $\Delta t=0.1$. For all mappings, we see a growth in the magnetic field. However, the growth rates differ as well as the saturation level, especially for the elliptical grid, where the saturation is higher compared to the other cases.

\subsection{Conservation Properties}
Conclusively, let us take a look at the conservation properties of the different time integrators on the deformed mapped grids.  
\begin{table}[htbp]
\centering
\caption{Weibel instability with perfect conductor boundary conditions: Maximum error in Gauss' law and in the total energy until time 500 for the semi-explicit and semi-implicit time integrators on various grids with time step $\Delta t = 0.1$ for the results in black and $\Delta t = 0.01$ for the results in \textcolor{gray}{gray}.}
\begin{tabular}{|c||c||l|l|l|l|}
\hline
&Method & Cartesian & Distorted & Cylindrical & Elliptical  \\
\hline
\multirow{3}{*}{Gauss} &HS & $ 1.8\cdot 10^{-11}$ & $ 5.2\cdot 10^{-10}$ & $ \color{gray}{ 6.3 \cdot 10^{-10}}$& $ \color{gray}{5.8 \cdot 10^{-10}}$\\
\cline{2-6}
&CEF & $ 1.8 \cdot 10^{-11}$ & $ 5.2 \cdot 10^{-10}$ & $ \color{gray}{6.4 \cdot 10^{-10}}$& $ \color{gray}{5.6 \cdot 10^{-10}}$ \\
\cline{2-6}
&DisGradE & $ 3.4 \cdot 10^{-4}$ & $  4.0 \cdot 10^{-4}$&  $ 2.1 \cdot 10^{-3}$& $ 4.7 \cdot 10^{-3}$\\
\hline
\multirow{3}{*} {Energy} &HS  & $ 1.8 \cdot 10^{-4}$ & $ 1.5 \cdot 10^{-4}$ &  $ \color{gray}{1.5 \cdot 10^{-6}} $& $ \color{gray}{2.0\cdot 10^{-6}}$\\
\cline{2-6}  
&CEF  &  $ 1.8 \cdot 10^{-4}$ & $ 3.7 \cdot 10^{-3}$ & $ \color{gray}{1.8 \cdot 10^{-6}}$& $\color{gray}{7.0\cdot 10^{-7}}$\\
\cline{2-6}  
&DisGradE & $ 2.0 \cdot 10^{-11}$ & $ 9.2 \cdot 10^{-12}$ & $ 5.0 \cdot 10^{-12}$& $  1.1\cdot 10^{-11}$\\
\hline
\end{tabular}
\label{table:weibelboundary}
\end{table}
Table \ref{table:weibelboundary} shows the conservation properties of the four time integrators until $T=500$ with a time step of $\Delta t = 0.1$ except for the semi-explicit schemes on the radial grids, where stability constraints restrict to a time step of $\Delta t = 0.01$. The distortion parameter is chosen as $\epsilon=0.05$.  We see the difference between the energy and the charge conserving methods. As expected, the implicit DisGradE  method conserves the total energy whereas for the semi-explicit HS and CEF schemes the energy is not conserved. Though, the energy error is bounded for the semi-explicit schemes. In contrast, the semi-explicit HS and CEF schemes conserve Gauss' law, which is not the case for the DisGradE method. Note that all conservation properties are up to the tolerance of the solver times the condition number of the mass matrices. 

\section{Conclusion and Outlook}
\label{sec:conclusion}
In this paper, we have investigated the natural boundary conditions of the weak formulation of Maxwell's equations and constructed new basis functions from clamped basis splines that form a discrete de Rham sequence. 
Furthermore, we have presented a fast and efficient preconditioner based on the eigenvalues of the mass matrix on a periodic hyperrectangle for the conjugate gradient solvers of the mass matrices of our system. 
Then, we have applied perfect conductor boundary conditions for the fields and reflecting boundary conditions for the particles enabling the use of domain deforming coordinate transformations such as cylindrical or elliptical mappings. 
In a simulation of the Weibel instability, which was inspired by a similar approach in \cite{chacon2019}, we have studied the effect of the deformed mapped grids and verified the compatibility of the boundary conditions with the conservation properties of the structure preserving discretisations. 

Following the recent publication of Toshniwal \& Hughes \cite{toshniwal2021isogeometric}, as future work, it would be possible to construct smooth spline basis functions that form a de Rham complex at the pole as described in \cite{patrizi2021isogeometric}. The implementation of such basis functions for the GEMPIC framework would enable the use of radial grids with a singularity.

\section{Acknowledgements}

Fruitful discussions with Hendrik Speelers are gratefully acknowledged.

\appendix
\section{Divergence Theorem}

\begin{prop}
For the integration by parts, we use the divergence theorem in the following forms:
\begin{itemize}
\item The standard form for a scalar function $g$ and a vector field $\Fb$,
{\footnotesize
\begin{align}\label{scalardivergencethm}
\int_\Omega \nabla \cdot ( \Fb g  ) \du\xb=\int_\Omega g (\nabla \cdot \Fb) \du\xb =-\int_\Omega  \nabla g \cdot  \Fb  \du\xb + \int_{\partial \Omega} g (\Fb \cdot \bn) \du\sigma, 
\end{align}}

\item The cross product form for vector fields, $\Fb, \Gb$,
{\footnotesize
\begin{align}\label{crossdivergencethm}
\int_\Omega \nabla \times ( \Fb \cdot \Gb) \du\xb =
\int_\Omega \Gb \cdot (\nabla \times \Fb) \du\xb -\int_\Omega (\nabla \times \Gb) \cdot \Fb  \du\xb =  \int_{\partial \Omega} (\Fb \times \Gb) \cdot \bn \du\sigma.
\end{align}}

\end{itemize}
\end{prop}